\documentclass{amsart}
\newif\ifcref\creftrue
\usepackage{amssymb,amsmath,stmaryrd,mathrsfs}
\usepackage[T1]{fontenc}

\def\definetac{\newif\iftac}    
\ifx\tactrue\undefined
  \definetac
  \ifx\state\undefined\tacfalse\else\tactrue\fi
\fi

\def\definebeamer{\newif\ifbeamer}
\ifx\beamertrue\undefined
  \definebeamer
  \ifx\uncover\undefined\beamerfalse\else\beamertrue\fi
\fi

\def\definecref{\newif\ifcref}
\ifx\creftrue\undefined
  \definecref
  \creffalse
\fi

\iftac\else\usepackage{amsthm}\fi
\usepackage{tikz,tikz-cd}
\usetikzlibrary{arrows}
\ifbeamer\else
  \usepackage{enumitem}
  \usepackage{xcolor}
  \definecolor{darkgreen}{rgb}{0,0.45,0} 
  \iftac\else\usepackage[pagebackref,colorlinks,citecolor=darkgreen,linkcolor=darkgreen]{hyperref}
  \renewcommand*{\backref}[1]{}
  \renewcommand*{\backrefalt}[4]{({%
      \ifcase #1 Not cited.%
            \or Cited on p.~#2%
            \else Cited on pp.~#2%
      \fi%
    })}\fi
\fi
\usepackage{mathtools}          
\usepackage{ifmtarg}            
\usepackage{eucal}
\usepackage{braket}             

\usepackage{url}                
\usepackage{xspace}             
\ifcref\usepackage{cleveref,aliascnt}\fi
\usepackage[status=draft,author=]{fixme}
\fxusetheme{color}
\usepackage{mathpartir}

\makeatletter
\let\ea\expandafter

\def\mdef#1#2{\ea\ea\ea\gdef\ea\ea\noexpand#1\ea{\ea\ensuremath\ea{#2}\xspace}}
\def\alwaysmath#1{\ea\ea\ea\global\ea\ea\ea\let\ea\ea\csname your@#1\endcsname\csname #1\endcsname
  \ea\def\csname #1\endcsname{\ensuremath{\csname your@#1\endcsname}\xspace}}

\newcount\foreachcount

\def\foreachletter#1#2#3{\foreachcount=#1
  \ea\loop\ea\ea\ea#3\@alph\foreachcount
  \advance\foreachcount by 1
  \ifnum\foreachcount<#2\repeat}

\def\foreachLetter#1#2#3{\foreachcount=#1
  \ea\loop\ea\ea\ea#3\@Alph\foreachcount
  \advance\foreachcount by 1
  \ifnum\foreachcount<#2\repeat}

\def\definescr#1{\ea\gdef\csname s#1\endcsname{\ensuremath{\mathscr{#1}}\xspace}}
\foreachLetter{1}{27}{\definescr}
\def\definecal#1{\ea\gdef\csname c#1\endcsname{\ensuremath{\mathcal{#1}}\xspace}}
\foreachLetter{1}{27}{\definecal}
\def\definebold#1{\ea\gdef\csname b#1\endcsname{\ensuremath{\mathbf{#1}}\xspace}}
\foreachLetter{1}{27}{\definebold}
\def\definebb#1{\ea\gdef\csname d#1\endcsname{\ensuremath{\mathbb{#1}}\xspace}}
\foreachLetter{1}{27}{\definebb}
\def\definefrak#1{\ea\gdef\csname f#1\endcsname{\ensuremath{\mathfrak{#1}}\xspace}}
\foreachletter{1}{9}{\definefrak}
\foreachletter{10}{27}{\definefrak}
\foreachLetter{1}{27}{\definefrak}
\def\definesf#1{\ea\gdef\csname i#1\endcsname{\ensuremath{\mathsf{#1}}\xspace}}
\foreachletter{1}{6}{\definesf}
\foreachletter{7}{14}{\definesf}
\foreachletter{15}{20}{\definesf}
\foreachletter{21}{27}{\definesf}
\foreachLetter{1}{27}{\definesf}
\def\definebar#1{\ea\gdef\csname #1bar\endcsname{\ensuremath{\overline{#1}}\xspace}}
\foreachLetter{1}{27}{\definebar}
\foreachletter{1}{8}{\definebar} 
\foreachletter{9}{15}{\definebar} 
\foreachletter{16}{27}{\definebar}
\def\definetil#1{\ea\gdef\csname #1til\endcsname{\ensuremath{\widetilde{#1}}\xspace}}
\foreachLetter{1}{27}{\definetil}
\foreachletter{1}{27}{\definetil}
\def\definehat#1{\ea\gdef\csname #1hat\endcsname{\ensuremath{\widehat{#1}}\xspace}}
\foreachLetter{1}{27}{\definehat}
\foreachletter{1}{27}{\definehat}
\def\definechk#1{\ea\gdef\csname #1chk\endcsname{\ensuremath{\widecheck{#1}}\xspace}}
\foreachLetter{1}{27}{\definechk}
\foreachletter{1}{27}{\definechk}
\def\defineul#1{\ea\gdef\csname u#1\endcsname{\ensuremath{\underline{#1}}\xspace}}
\foreachLetter{1}{27}{\defineul}
\foreachletter{1}{27}{\defineul}

\def\autofmt@n#1\autofmt@end{\mathrm{#1}}
\def\autofmt@b#1\autofmt@end{\mathbf{#1}}
\def\autofmt@d#1#2\autofmt@end{\mathbb{#1}\mathsf{#2}}
\def\autofmt@c#1#2\autofmt@end{\mathcal{#1}\mathit{#2}}
\def\autofmt@s#1#2\autofmt@end{\mathscr{#1}\mathit{#2}}
\def\autofmt@i#1\autofmt@end{\mathsf{#1}}
\def\autofmt@f#1\autofmt@end{\mathfrak{#1}}
\def\autofmt@u#1\autofmt@end{\underline{\smash{\mathsf{#1}}}}
\def\autofmt@U#1\autofmt@end{\underline{\underline{\smash{\mathsf{#1}}}}}
\def\autofmt@h#1\autofmt@end{\widehat{#1}}
\def\autofmt@r#1\autofmt@end{\overline{#1}}
\def\autofmt@t#1\autofmt@end{\widetilde{#1}}
\def\autofmt@k#1\autofmt@end{\check{#1}}

\def\auto@drop#1{}
\def\autodef#1{\ea\ea\ea\@autodef\ea\ea\ea#1\ea\auto@drop\string#1\autodef@end}
\def\@autodef#1#2#3\autodef@end{%
  \ea\def\ea#1\ea{\ea\ensuremath\ea{\csname autofmt@#2\endcsname#3\autofmt@end}\xspace}}
\def\autodefs@end{blarg!}
\def\autodefs#1{\@autodefs#1\autodefs@end}
\def\@autodefs#1{\ifx#1\autodefs@end%
  \def\autodefs@next{}%
  \else%
  \def\autodefs@next{\autodef#1\@autodefs}%
  \fi\autodefs@next}

\DeclareSymbolFont{bbold}{U}{bbold}{m}{n}
\DeclareSymbolFontAlphabet{\mathbbb}{bbold}

\newcommand{\done}{\ensuremath{\mathbbb{1}}\xspace}
\newcommand{\dtwo}{\ensuremath{\mathbbb{2}}\xspace}

\mdef\delbar{\overline{\partial}}

\mdef\hf{\textstyle\frac12 }
\mdef\thrd{\textstyle\frac13 }
\mdef\qtr{\textstyle\frac14 }

\newcommand{\op}{^{\mathrm{op}}}

\mdef\Id{\mathrm{Id}}
\mdef\id{\mathrm{id}}
\alwaysmath{ell}
\alwaysmath{infty}

\alwaysmath{odot}
\def\frc#1/#2.{\frac{#1}{#2}}   
\mdef\ten{\mathrel{\otimes}}

\mdef\sqten{\mathrel{\boxtimes}}

\DeclareFontFamily{U}{min}{}
\DeclareFontShape{U}{min}{m}{n}{<-> udmj30}{}

\DeclareFontFamily{U}{mathb}{\hyphenchar\font45}
\DeclareFontShape{U}{mathb}{m}{n}{
      <5> <6> <7> <8> <9> <10> gen * mathb
      <10.95> mathb10 <12> <14.4> <17.28> <20.74> <24.88> mathb12
      }{}
\DeclareSymbolFont{mathb}{U}{mathb}{m}{n}
\DeclareFontSubstitution{U}{mathb}{m}{n}
\DeclareFontFamily{U}{mathx}{\hyphenchar\font45}
\DeclareFontShape{U}{mathx}{m}{n}{
      <5> <6> <7> <8> <9> <10>
      <10.95> <12> <14.4> <17.28> <20.74> <24.88>
      mathx10
      }{}
\DeclareSymbolFont{mathx}{U}{mathx}{m}{n}
\DeclareFontSubstitution{U}{mathx}{m}{n}
\DeclareMathSymbol{\dotplus}       {2}{mathb}{"00}
\DeclareMathSymbol{\dotdiv}        {2}{mathb}{"01}
\DeclareMathSymbol{\dottimes}      {2}{mathb}{"02}
\DeclareMathSymbol{\divdot}        {2}{mathb}{"03}
\DeclareMathSymbol{\udot}          {2}{mathb}{"04}
\DeclareMathSymbol{\square}        {2}{mathb}{"05}
\DeclareMathSymbol{\Asterisk}      {2}{mathb}{"06}
\DeclareMathSymbol{\bigast}        {1}{mathb}{"06}
\DeclareMathSymbol{\coAsterisk}    {2}{mathb}{"07}
\DeclareMathSymbol{\bigcoast}      {1}{mathb}{"07}
\DeclareMathSymbol{\circplus}      {2}{mathb}{"08}
\DeclareMathSymbol{\pluscirc}      {2}{mathb}{"09}
\DeclareMathSymbol{\convolution}   {2}{mathb}{"0A}
\DeclareMathSymbol{\divideontimes} {2}{mathb}{"0B}
\DeclareMathSymbol{\blackdiamond}  {2}{mathb}{"0C}
\DeclareMathSymbol{\sqbullet}      {2}{mathb}{"0D}
\DeclareMathSymbol{\bigstar}       {2}{mathb}{"0E}
\DeclareMathSymbol{\bigvarstar}    {2}{mathb}{"0F}
\DeclareMathSymbol{\corresponds}   {3}{mathb}{"1D}
\DeclareMathSymbol{\boxleft}      {2}{mathb}{"68}
\DeclareMathSymbol{\boxright}     {2}{mathb}{"69}
\DeclareMathSymbol{\boxtop}       {2}{mathb}{"6A}
\DeclareMathSymbol{\boxbot}       {2}{mathb}{"6B}
\DeclareMathSymbol{\updownarrows}          {3}{mathb}{"D6}
\DeclareMathSymbol{\downuparrows}          {3}{mathb}{"D7}
\DeclareMathSymbol{\Lsh}                   {3}{mathb}{"E8}
\DeclareMathSymbol{\Rsh}                   {3}{mathb}{"E9}
\DeclareMathSymbol{\dlsh}                  {3}{mathb}{"EA}
\DeclareMathSymbol{\drsh}                  {3}{mathb}{"EB}
\DeclareMathSymbol{\looparrowdownleft}     {3}{mathb}{"EE}
\DeclareMathSymbol{\looparrowdownright}    {3}{mathb}{"EF}
\DeclareMathSymbol{\curvearrowleftright}   {3}{mathb}{"F2}
\DeclareMathSymbol{\curvearrowbotleft}     {3}{mathb}{"F3}
\DeclareMathSymbol{\curvearrowbotright}    {3}{mathb}{"F4}
\DeclareMathSymbol{\curvearrowbotleftright}{3}{mathb}{"F5}
\DeclareMathSymbol{\leftsquigarrow}        {3}{mathb}{"F8}
\DeclareMathSymbol{\rightsquigarrow}       {3}{mathb}{"F9}
\DeclareMathSymbol{\leftrightsquigarrow}   {3}{mathb}{"FA}
\DeclareMathSymbol{\lefttorightarrow}      {3}{mathb}{"FC}
\DeclareMathSymbol{\righttoleftarrow}      {3}{mathb}{"FD}
\DeclareMathSymbol{\uptodownarrow}         {3}{mathb}{"FE}
\DeclareMathSymbol{\downtouparrow}         {3}{mathb}{"FF}
\DeclareMathSymbol{\varhash}       {0}{mathb}{"23}
\DeclareMathSymbol{\sqSubset}       {3}{mathb}{"94}
\DeclareMathSymbol{\sqSupset}       {3}{mathb}{"95}
\DeclareMathSymbol{\nsqSubset}      {3}{mathb}{"96}
\DeclareMathSymbol{\nsqSupset}      {3}{mathb}{"97}
\DeclareMathAccent{\widecheck}    {0}{mathx}{"71}

\DeclareMathOperator\ob{ob}

\newcommand{\too}[1][]{\ensuremath{\overset{#1}{\longrightarrow}}}
\newcommand{\ot}{\ensuremath{\leftarrow}}

\let\toto\rightrightarrows
\let\into\hookrightarrow

\mdef\we{\overset{\sim}{\longrightarrow}}
\mdef\leftwe{\overset{\sim}{\longleftarrow}}

\let\epi\twoheadrightarrow

\let\xto\xrightarrow
\let\xot\xleftarrow
\def\rightarrowtailfill@{\arrowfill@{\Yright\joinrel\relbar}\relbar\rightarrow}
\newcommand\xrightarrowtail[2][]{\ext@arrow 0055{\rightarrowtailfill@}{#1}{#2}}

\def\twoheadrightarrowfill@{\arrowfill@{\relbar\joinrel\relbar}\relbar\twoheadrightarrow}
\newcommand\xtwoheadrightarrow[2][]{\ext@arrow 0055{\twoheadrightarrowfill@}{#1}{#2}}

\def\slashedarrowfill@#1#2#3#4#5{%
  $\m@th\thickmuskip0mu\medmuskip\thickmuskip\thinmuskip\thickmuskip
   \relax#5#1\mkern-7mu%
   \cleaders\hbox{$#5\mkern-2mu#2\mkern-2mu$}\hfill
   \mathclap{#3}\mathclap{#2}%
   \cleaders\hbox{$#5\mkern-2mu#2\mkern-2mu$}\hfill
   \mkern-7mu#4$%
}
\def\rightslashedarrowfill@{%
  \slashedarrowfill@\relbar\relbar\mapstochar\rightarrow}
\newcommand\xslashedrightarrow[2][]{%
  \ext@arrow 0055{\rightslashedarrowfill@}{#1}{#2}}
\mdef\hto{\xslashedrightarrow{}}
\mdef\htoo{\xslashedrightarrow{\quad}}

\def\xiso#1{\mathrel{\mathrlap{\smash{\xto[\smash{\raisebox{1.3mm}{$\scriptstyle\sim$}}]{#1}}}\hphantom{\xto{#1}}}}
\def\toiso{\xto{\smash{\raisebox{-.5mm}{$\scriptstyle\sim$}}}}

\def\jd#1{\@jd#1\ej}
\def\@jd#1|-#2\ej{\@@jd#1,,\;\vdash\;\left(#2\right)}
\def\@@jd#1,{\@ifmtarg{#1}{\let\next=\relax}{\left(#1\right)\let\next=\@@@jd}\next}
\def\@@@jd#1,{\@ifmtarg{#1}{\let\next=\relax}{,\,\left(#1\right)\let\next=\@@@jd}\next}
\def\jdm#1{\@jdm#1\ej}
\def\@jdm#1|-#2\ej{\@@jd#1,,\\\vdash\;\left(#2\right)}

\long\def\my@drawfill#1#2;{%
\@skipfalse
\fill[#1,draw=none] #2;
\@skiptrue
\draw[#1,fill=none] #2;
}
\newif\if@skip
\newcommand{\skipit}[1]{\if@skip\else#1\fi}
\newcommand{\drawfill}[1][]{\my@drawfill{#1}}

\newcounter{nodemaker}
\newcommand{\twocell}[2][]{%
  \global\edef\mynodeone{twocell\arabic{nodemaker}}%
  \stepcounter{nodemaker}%
  \global\edef\mynodetwo{twocell\arabic{nodemaker}}%
  \stepcounter{nodemaker}%
  \ar[#2,phantom,shift left=3,""{name=\mynodeone}]%
  \ar[#2,phantom,shift right=3,""'{name=\mynodetwo}]%
  \ar[Rightarrow,from=\mynodeone,to=\mynodetwo,"{#1}"]%
}

\newcommand{\drpullback}[1][dr]{\ar[#1,phantom,near start,"\lrcorner"]}

\newif\ifhyperref
\@ifpackageloaded{hyperref}{\hyperreftrue}{\hyperreffalse}
\iftac
  \let\your@state\state
  \def\state#1{\my@state#1}
  \def\my@state#1.{\gdef\currthmtype{#1}\your@state{#1.}}
  \let\your@staterm\staterm
  \def\staterm#1{\my@staterm#1}
  \def\my@staterm#1.{\gdef\currthmtype{#1}\your@staterm{#1.}}
  \let\@defthm\newtheorem
  \def\switchtotheoremrm{\let\@defthm\newtheoremrm}
  \def\defthm#1#2#3{\@defthm{#1}{#2}} 
  \let\your@section\section
  \def\section{\gdef\currthmtype{section}\your@section}
  \let\your@figure\figure
  \def\figure{\gdef\currthmtype{Figure}\your@figure}
  \def\currthmtype{}
  \ifhyperref
    \def\autoref#1{\ref*{label@name@#1}~\ref{#1}}
  \else
    \def\autoref#1{\ref{label@name@#1}~\ref{#1}}
  \fi
  \AtBeginDocument{%
    \let\old@label\label%
    \def\label#1{%
      {\let\your@currentlabel\@currentlabel%
        \edef\@currentlabel{\currthmtype}%
        \old@label{label@name@#1}}%
      \old@label{#1}}
  }
  \let\cref\autoref
\else\ifcref
  \def\defthm#1#2#3{%
    \newaliascnt{#1}{thm}
    \newtheorem{#1}[#1]{#2}
    \aliascntresetthe{#1}
    \crefname{#1}{#2}{#3}
  }
\else
  \ifhyperref
    \def\defthm#1#2#3{
      \newtheorem{#1}{#2}[section]%
      \expandafter\def\csname #1autorefname\endcsname{#2}%
      \expandafter\let\csname c@#1\endcsname\c@thm}
  \else
    \def\defthm#1#2#3{\newtheorem{#1}[thm]{#2}} 
    \ifx\SK@label\undefined\let\SK@label\label\fi
    \let\old@label\label
    \let\your@thm\@thm
    \def\@thm#1#2#3{\gdef\currthmtype{#3}\your@thm{#1}{#2}{#3}}
    \def\currthmtype{}
    \def\label#1{{\let\your@currentlabel\@currentlabel\def\@currentlabel%
        {\currthmtype~\your@currentlabel}%
        \SK@label{#1@}}\old@label{#1}}
    \def\autoref#1{\ref{#1@}}
  \fi
  \let\cref\autoref
\fi\fi

\newtheorem{thm}{Theorem}[section]
\ifcref
  \crefname{thm}{Theorem}{Theorems}
\else
  
\fi
\defthm{cor}{Corollary}{Corollaries}
\defthm{prop}{Proposition}{Propositions}
\defthm{lem}{Lemma}{Lemmas}
\defthm{sch}{Scholium}{Scholia}
\defthm{assume}{Assumption}{Assumptions}
\defthm{claim}{Claim}{Claims}
\defthm{conj}{Conjecture}{Conjectures}
\defthm{hyp}{Hypothesis}{Hypotheses}
\iftac\switchtotheoremrm\else\theoremstyle{definition}\fi
\defthm{defn}{Definition}{Definitions}
\defthm{notn}{Notation}{Notations}
\iftac\switchtotheoremrm\else\theoremstyle{remark}\fi
\defthm{rmk}{Remark}{Remarks}
\defthm{eg}{Example}{Examples}
\defthm{egs}{Examples}{Examples}
\defthm{ex}{Exercise}{Exercises}
\defthm{ceg}{Counterexample}{Counterexamples}

\ifcref
  \crefformat{section}{\S#2#1#3}
  \Crefformat{section}{Section~#2#1#3}
  \crefrangeformat{section}{\S\S#3#1#4--#5#2#6}
  \Crefrangeformat{section}{Sections~#3#1#4--#5#2#6}
  \crefmultiformat{section}{\S\S#2#1#3}{ and~#2#1#3}{, #2#1#3}{ and~#2#1#3}
  \Crefmultiformat{section}{Sections~#2#1#3}{ and~#2#1#3}{, #2#1#3}{ and~#2#1#3}
  \crefrangemultiformat{section}{\S\S#3#1#4--#5#2#6}{ and~#3#1#4--#5#2#6}{, #3#1#4--#5#2#6}{ and~#3#1#4--#5#2#6}
  \Crefrangemultiformat{section}{Sections~#3#1#4--#5#2#6}{ and~#3#1#4--#5#2#6}{, #3#1#4--#5#2#6}{ and~#3#1#4--#5#2#6}
  \crefformat{appendix}{Appendix~#2#1#3}
  \Crefformat{appendix}{Appendix~#2#1#3}
  \crefrangeformat{appendix}{Appendices~#3#1#4--#5#2#6}
  \Crefrangeformat{appendix}{Appendices~#3#1#4--#5#2#6}
  \crefmultiformat{appendix}{Appendices~#2#1#3}{ and~#2#1#3}{, #2#1#3}{ and~#2#1#3}
  \Crefmultiformat{appendix}{Appendices~#2#1#3}{ and~#2#1#3}{, #2#1#3}{ and~#2#1#3}
  \crefrangemultiformat{appendix}{Appendices~#3#1#4--#5#2#6}{ and~#3#1#4--#5#2#6}{, #3#1#4--#5#2#6}{ and~#3#1#4--#5#2#6}
  \Crefrangemultiformat{appendix}{Appendices~#3#1#4--#5#2#6}{ and~#3#1#4--#5#2#6}{, #3#1#4--#5#2#6}{ and~#3#1#4--#5#2#6}
  \crefformat{subappendix}{\S#2#1#3}
  \Crefformat{subappendix}{Section~#2#1#3}
  \crefrangeformat{subappendix}{\S\S#3#1#4--#5#2#6}
  \Crefrangeformat{subappendix}{Sections~#3#1#4--#5#2#6}
  \crefmultiformat{subappendix}{\S\S#2#1#3}{ and~#2#1#3}{, #2#1#3}{ and~#2#1#3}
  \Crefmultiformat{subappendix}{Sections~#2#1#3}{ and~#2#1#3}{, #2#1#3}{ and~#2#1#3}
  \crefrangemultiformat{subappendix}{\S\S#3#1#4--#5#2#6}{ and~#3#1#4--#5#2#6}{, #3#1#4--#5#2#6}{ and~#3#1#4--#5#2#6}
  \Crefrangemultiformat{subappendix}{Sections~#3#1#4--#5#2#6}{ and~#3#1#4--#5#2#6}{, #3#1#4--#5#2#6}{ and~#3#1#4--#5#2#6}
  \crefname{part}{Part}{Parts}
  \crefname{figure}{Figure}{Figures}
\fi

\iftac
  \let\qed\endproof
  \let\your@endproof\endproof
  \def\my@endproof{\your@endproof}
  \def\endproof{\my@endproof\gdef\my@endproof{\your@endproof}}
  \def\qedhere{\tag*{\endproofbox}\gdef\my@endproof{\relax}}
\fi

\iftac
  \def\pr@@f[#1]{\subsubsection*{\sc #1.}}
\fi

\def\thmqedhere{\expandafter\csname\csname @currenvir\endcsname @qed\endcsname}

\ifbeamer\else

\fi

\ifbeamer\else
  \setitemize[1]{leftmargin=2em}
  \setenumerate[1]{leftmargin=*}
\fi

\iftac
  \let\c@equation\c@subsection
\else
  \let\c@equation\c@thm
\fi
\numberwithin{equation}{section}

\ifcref\else
  \@ifpackageloaded{mathtools}{\mathtoolsset{showonlyrefs,showmanualtags}}{}
\fi

\alwaysmath{alpha}
\alwaysmath{beta}
\alwaysmath{gamma}
\alwaysmath{Gamma}
\alwaysmath{delta}
\alwaysmath{Delta}
\alwaysmath{epsilon}
\mdef\ep{\varepsilon}
\alwaysmath{zeta}
\alwaysmath{eta}
\alwaysmath{theta}
\alwaysmath{Theta}
\alwaysmath{iota}
\alwaysmath{kappa}
\alwaysmath{lambda}
\alwaysmath{Lambda}
\alwaysmath{mu}
\alwaysmath{nu}
\alwaysmath{xi}
\alwaysmath{pi}
\alwaysmath{rho}
\alwaysmath{sigma}
\alwaysmath{Sigma}
\alwaysmath{tau}
\alwaysmath{upsilon}
\alwaysmath{Upsilon}
\alwaysmath{phi}
\alwaysmath{Pi}
\alwaysmath{Phi}
\mdef\ph{\varphi}
\alwaysmath{chi}
\alwaysmath{psi}
\alwaysmath{Psi}
\alwaysmath{omega}
\alwaysmath{Omega}
\let\al\alpha
\let\be\beta
\let\gm\gamma

\let\si\sigma

\let\ze\zeta
\let\th\theta

\makeatother

\title{The derivator of setoids}
\author{Michael Shulman}
\autodefs{\iSet\cCat\cCAT\nHo\nsSet\ntw\ncod\ndom\iProp\iGpd\iEGpd\iRGpd\iContr\iSpace\iUGpd}
\mdef\cECat{\mathcal{EC}\mathit{at}}
\mdef\cUCat{\mathcal{UC}\mathit{at}}
\mdef\sex{\iSet_{\mathsf{ex}}}
\mdef\eex{\sE_{\mathsf{ex}}}
\mdef\sreg{\iSet_{\mathsf{reg}}}
\mdef\ereg{\sE_{\mathsf{reg}}}
\mdef\spos{\iSet_{\mathsf{pos}}}
\mdef\epos{\sE_{\mathsf{pos}}}
\mdef\hom{\mathsf{Hom}}
\mdef\cchom{\mathsf{Hom}_{\mathsf{cc}}}
\def\ob#1{#1_0}
\def\incl#1{\iota_{#1}}
\let\D\sD
\let\T\sT
\let\W\cW
\def\tint{\mathord{\textstyle\int}}
\def\odt{\mathbin{\widetilde{\odot}}}
\mdef\etatil{\widetilde{\eta}}
\def\ana{_{\mathsf{ana}}}
\mdef\sana{\iSpace\ana}
\date{\today}
\thanks{This material is based upon work supported by the Air Force Office of Scientific Research under award
number FA9550-21-1-0009.}
\begin{document}


\begin{abstract}
  Without the axiom of choice, the free exact completion of the category of sets (i.e.\ the category of setoids) may not be complete or cocomplete.
  We will show that nevertheless, it can be enhanced to a derivator: the formal structure of categories of diagrams related by Kan extension functors.
  Moreover, this derivator is the free cocompletion of a point in a class of ``1-truncated derivators'' (which behave like a 1-category rather than a higher category).
  
  In classical mathematics, the free cocompletion of a point relative to all derivators is the homotopy theory of spaces.
  Thus, if there is a homotopy theory that can be shown to have this universal property constructively, its 1-truncation must contain not only sets, but also setoids.
  This suggests that either setoids are an unavoidable aspect of constructive homotopy theory, or more radical modifications to the notion of homotopy theory are needed.
\end{abstract}

\maketitle
\setcounter{tocdepth}{1}
\tableofcontents

\section{Introduction}
\label{sec:introduction}

Can homotopy theory be developed in constructive mathematics, or even in ZF set theory without the axiom of choice?
%
Recently this question has begun to attract more attention, due partly to the rise of interest in Homotopy Type Theory and Univalent Foundations~\cite{hottbook}.
The latter is a \emph{constructive} type theory whose first model was nevertheless relentlessly \emph{classical}, using the Kan--Quillen model category of simplicial sets~\cite{klv:ssetmodel}.
Since then, constructive models of homotopy type theory have been found in categories of cubical sets~\cite{bch:tt-cubical,bch:univalence-cubical,cchm:cubicaltt,abcfhl:cart-cube,accrs:equivariant}, and the model category of simplicial sets has been developed constructively~\cite{henry:constr-sset,gss:constr-sset,gh:constr-sset-uf,ghss:eff-kan}, though not quite to the point of strictly modeling type theory.

In particular, there are now at least two constructive homotopy theories --- the aforementioned simplicial sets and the equivariant cartesian cubical sets of~\cite{accrs:equivariant} --- that can \emph{classically} be shown to present the homotopy theory of spaces.
However, it is not known whether they are \emph{constructively} equivalent to \emph{each other}. 
Thus one may naturally wonder: if they are not equivalent, which is the ``correct'' constructive homotopy theory of spaces?\footnote{By ``space'' we mean some combinatorial notion of $\infty$-groupoid.  It is probably not reasonable to expect a theory of $\infty$-groupoids to be constructively equivalent to the homotopy theory of \emph{topological} spaces, as continuous functions are much less flexible constructively than classically.}
Or, perhaps, are they both ``incorrect''?
What does ``correct'' even mean?

In fact, both of these homotopy theories have a property that at first may seem peculiar: their 1-truncations (meaning their subcategory of homotopy 0-types) are \emph{not} equivalent to the category of (constructive) sets that we started from.
The 1-truncation of simplicial sets appears to be equivalent to the free exact completion \sex of \iSet~\cite{cm:ex-lex}, a.k.a.\ the category of ``setoids'' (Simon Henry, personal communication).
The 1-truncation of equivariant cartesian cubical sets may not be equivalent to \sex (Andrew Swan, personal communication), but neither is it equivalent to \iSet.
This is a significant departure from both classical mathematics and homotopy type theory, in which sets can be regarded, up to equivalence, as homotopy 0-types.
(Note that the inclusion $\iSet\into \sex$ is an equivalence if and only if the axiom of choice holds.)

In particular, this means that when homotopy type theory is interpreted in one of these constructive model categories, its internally-defined ``sets'' will be interpreted in the model as some kind of \emph{setoid} rather than as actual sets.
This is somewhat disturbing for the prospect of constructive applications of homotopy type theory and its semantics.
At a stretch, one might even regard it as evidence for the \emph{incorrectness} of \emph{both} of these model categories.

In this paper we propose one possible correctness criterion for a constructive homotopy theory of spaces.
Moreover, we provide some evidence that, the foregoing remarks notwithstanding, the 1-truncation of \emph{any} theory satisfying this criterion \emph{must} contain at least \sex, not just \iSet.
In a moment we will discuss possible interpretations of this fact, but first let us explain the criterion and the evidence.


Classically, the homotopy theory of spaces has a universal property: it is the free cocomplete $(\infty,1)$-category generated by a point~\cite[5.1.5.6]{lurie:higher-topoi}, just as \iSet is the free cocomplete 1-category generated by a point.
However, this is somewhat circular as a characterization, since an $(\infty,1)$-category is defined to have \emph{spaces} as hom-objects.\footnote{To be sure, not all definitions of $(\infty,1)$-category explicitly incorporate hom-spaces.
  But the question of the correct constructive definition of $(\infty,1)$-category seems likely to be at least as difficult as that of the correct constructive definition of $\infty$-groupoid, i.e.\ homotopy space.}
One possible way around this would be to work with \emph{presentations} of $(\infty,1)$-categories using 1-categorical structures such as Quillen model categories.
However, universal properties \emph{of} $(\infty,1)$-categories (as opposed to universal properties of objects \emph{in} an $(\infty,1)$-category) are hard to express at this level --- indeed, this is one of the main reasons for the recent explicit use of $(\infty,1)$-categories instead of model categories in applications such as~\cite{lurie:higher-topoi}.
Moreover, although in classical mathematics most interesting complete and cocomplete $(\infty,1)$-categories (including all locally presentable ones) can be presented by model categories, we ought not to assume \textit{a priori} that this will still be the case constructively.

Instead, we can work with a 1-categorical \emph{quotient} of an $(\infty,1)$-category.
The ordinary homotopy category, obtained by identifying equivalent pairs of parallel morphisms, is too coarse for this purpose; but an enhancement of it (due to Heller~\cite{heller:htpythys}, Grothendieck~\cite{grothendieck:derivateurs}, and Franke~\cite{franke:triang}) turns out to be sufficient.
Namely, given a complete and cocomplete $(\infty,1)$-category \sC, we consider the homotopy 1-categories of the functor $(\infty,1)$-categories $\sC^A$ for all small 1-categories $A$, together with the restriction functors relating them and their left and right adjoints (homotopy Kan extensions).
This structure is nowadays called a \emph{derivator} (after Grothendieck), and it retains a surprising amount of information about \sC.

In particular, Heller~\cite{heller:htpythys} and Cisinski~\cite{cisinski:presheaves} have shown, in classical mathematics, that the derivator \iSpace of spaces is the free cocompletion of a point.
This means that for any other derivator \D, the category of cocontinuous morphisms $\iSpace\to\D$ (those that preserve the ``formal left Kan extensions'' included in the structure of a derivator) is equivalent to the category $\D(\done)$ of ``diagrams of shape $\done$'' in \D (i.e.\ ``objects of \D''; here $\done$ denotes the terminal category).
Although a derivator is intuitively a homotopical, i.e.\ $(\infty,1)$-categorical, object, formally this universal property lives at the same categorical level as the universal property of \iSet: derivators, like 1-categories, form a 2-category, and the universal property is an equivalence involving hom-categories therein.
In the words of Cisinski~\cite{cisinski:deriv-comment}:

\begin{quote}
  \small
  This provides a first argument that the usual homotopy theory of simplicial sets plays a central role\dots and for this, we didn't take for granted that homotopy types should be that important: its universal property is formulated with category theory only.\dots derivators provide a truncated version of higher category theory which gives us the language to characterize higher category theory using only usual category theory, without any emphasis on any particular model (in fact, without assuming we even know any).
\end{quote}

Thus, a natural correctness criterion for a constructive homotopy theory of spaces would be that it defines a derivator \iSpace that is the free cocompletion of a point.

Of course, it is not \textit{a priori} clear that such a derivator even exists in constructive mathematics.
We will not attempt to construct one in this paper.
Instead, we will attempt to understand how \iSpace \emph{would} behave, if it exists, by studying derivators that ought to be \emph{localizations} of it.
By this we mean derivators that should be obtained from \iSpace by universally inverting some class of morphisms among cocontinuous morphisms, although in good situations this equivalent to being a reflective subcategory of \iSpace (a \emph{reflective localization}).

Classically, \iSpace has many interesting reflective localizations, such as those that invert some set of prime numbers.
More relevantly for us, for all integers $n\ge -2$ it has a reflective localization $\iSpace_n$ consisting of homotopy $n$-types.
In particular, $\iSpace_0$ is just the category $\iSet$ of sets (regarded as a derivator), while $\iSpace_{-1}$ is the poset \iProp of truth values (which, classically, is the two-element lattice) and $\iSpace_{-2}$ is the terminal derivator.
Moreover, each $\iSpace_n$ is the free cocompletion of a point in the world of ``$(n\mathord+1)$-truncated derivators'' --- those that behave like $(n\mathord+1,1)$-categories rather than $(\infty,1)$-categories.%
\footnote{These ``$n$-truncated derivators'' are distinct from the ``$n$-derivators'' of~\cite{raptis:hderiv}.
  The former are 1-derivators (in the terminology of~\cite{raptis:hderiv}) that act as if they arose from an $(n,1)$-category, while the latter generalize the definition of derivator to use $n$-categories in place of 1-categories.}
In particular, this universal property for \iSet generalizes its ordinary one, 
giving it a mapping property into all 1-truncated derivators, not just those that arise from 1-categories.

In this paper we will exhibit, in constructive mathematics, derivators $\iSpace_n$ that have this universal property for $n=0,-1,-2$.
In fact, for $n=0$ and $-1$ (and thus presumably for all $n\ge -1$) the notion of ``$(n\mathord+1)$-truncated derivator'' multifurcates constructively into several different notions, with several different corresponding localizations.

For one natural notion of ``1-truncated derivator'', we find that \iSet is the free cocompletion of a point.
However, there are intuitively ``1-categorical'' derivators that are not 1-truncated in this sense.
Notably, we will show that for any complete category \sE having small coproducts preserved by pullback, its exact completion \eex can be enhanced to a derivator, which is \emph{not} ``1-truncated'' in the naive \iSet-based sense.
There is a weaker notion of 1-truncatedness that does encompass these examples, but in this world \iSet is no longer the free cocompletion of the point: instead that role is taken by \sex.\footnote{It is unclear exactly how this universal property of the \emph{derivator} \sex is related to the usual universal property of the \emph{category} \sex.
  But it is reminiscent of the result of~\cite[Corollary to Lemma 4.1]{carboni:free-constr} that classically, the free exact completion of the small-coproduct completion of a small category is equivalent to its presheaf category, i.e.\ its free cocompletion.
  (Note that \iSet is the free small-coproduct completion of a point, as well as the free cocompletion of a point.)}
There is also an intermediate notion of ``1-truncatedness'', whose free cocompletion of a point is a derivator version of \sreg, the free regular completion of \iSet.
We will refer to these three notions of 1-truncatedness as being \emph{\iSet-local}, \emph{\sex-local}, and \emph{\sreg-local} respectively.

A similar thing happens one dimension down: in addition to the lattice \iProp, we have a derivator version of \spos, the preorder reflection of \iSet.
Each of them is the free cocompletion of a point in its corresponding world of local derivators.

The class of \sex-local derivators is broader than that of \iSet-local ones, and in particular there is a cocontinuous map of derivators $\sex \to \iSet$ but not conversely.
Thus, if both were realized as reflective subcategories of \iSpace, then \sex would be the larger one.
This provides our evidence that if a free cocompletion of a point exists constructively, its 1-truncation must involve \sex and not just \iSet.%
\footnote{There is the possibility that this 1-truncation could be something even larger than \sex.
  It is not clear whether $(\sex)_{\mathsf{ex}}$ can be made into a derivator at all, but if it could be then it would be one possible candidate.
  In addition, the 1-truncation of cubical sets may also be larger than \sex (Andrew Swan, personal communication), so it is another possibility.}

I can think of at least three responses to this observation.
The first is to bite the bullet and accept that the correct homotopy theory of spaces is constructively the ``$\infty$-exact completion'' of \iSet, and in particular its 0-truncated objects are setoids rather than sets.
Thus, when applying homotopy theory constructively, we would be forced to use setoids, either exclusively or in tandem with sets.

This may be satisfying if our motivations for constructivity are purely philosophical.
Indeed, some constructivist schools start from a foundation whose primitive objects are not sets but some kind of ``pre-set'' or ``type'' that lacks quotients entirely, such as some formalizations of Bishop's constructive mathematics~\cite{bb:constr-analysis} or Martin-L\"{o}f's original constructive type theory~\cite{martinlof:itt}.
In this case, if ``the category of sets'' is to be exact, it \emph{must} be defined as a free exact completion of the category of pre-sets, and so the appearance of an exact completion is entirely unproblematic.\footnote{Relatedly, note that the model category of simplicial objects constructed in~\cite{ghss:eff-kan} requires only a category with finite limits and extensive countable coproducts.}

However, if we also care about categorical semantics, the appearance of setoids is troubling.
When interpreting constructive mathematics internally in a category, it is the sets, not the setoids, that correspond to objects of that category.
If our category of interest happens itself to be an exact completion of some other category, we might be able to interpret our mathematics in the latter, with the former category appearing as the exact completion of the latter.
However, although some important categories are exact completions (such as some presheaf toposes and realizability toposes), many are not (such as most sheaf toposes), so this approach cannot work for them.
This is related to the problem of constructing ``realizability higher toposes'' whose underlying 1-topos is an ordinary realizability topos~\cite{uemura:cubical-asm,su:ct-cubasm}.

Another problem with exact completions is that they destroy impredicativity: even if \iSet has a subobject classifier, \sex generally will not.
Again, a philosophical predicativist may be unbothered by this, but it is disconcerting to \emph{choose} to work with an impredicative category \iSet and nevertheless be forced into the predicative \sex as soon as we start trying to do homotopy theory.


The second response is to reject our proposed ``correctness criterion'' for the homotopy theory of spaces.
And indeed, there are obvious grounds on which to do so.
Namely, our notion of derivator is based on small categories and \emph{functors} between them; but there are good arguments that in the absence of the axiom of choice, the correct notion of morphism between categories is instead that of an \emph{anafunctor}~\cite{makkai:avoiding-choice,bartels:hgt,roberts:ana}.
This suggests that we should instead be considering ``ana-derivators'' defining using anafunctors.
In that world, it might be the case that the free cocompletion of a point consists of spaces and anafunctors between them, and has \iSet as its 1-truncation.

However, there are difficulties involved in making this work.
Already for categories, it is impossible to prove even in ZF set theory that the bicategory of categories and anafunctors is locally small, cartesian closed, or complete~\cite{ak:nonsmall-ana}.
(There are much weaker axioms than AC that suffice for local smallness and cartesian closure, such as SCSA~\cite{makkai:avoiding-choice} and WISC~\cite{roberts:ana}, but their constructive status is arguable, and it is unclear whether they imply completeness as well.)
It seems likely that similar problems would arise in building a derivator out of 1-groupoids and anafunctors, let alone $\infty$-groupoids and $\infty$-anafunctors.

It may be more feasible to construct only a \emph{left} derivator of groupoids and anafunctors, which has colimits but not limits.
However, there are applications for which this would be insufficient; for instance, defining and constructing stacks requires taking limits over infinite sieves to define categories of descent data.

Finally, the third response is to reject the whole idea of \emph{defining} spaces constructively out of sets, and instead \emph{start} from a foundational theory such as homotopy type theory~\cite{hottbook}, in which spaces are primitive objects.
(Note that ``computably'' constructive flavors of homotopy type theory are also now available, such as the cubical type theories of~\cite{cchm:cubicaltt,abcfhl:cart-cube}.)
This allows ``sets'' to be \emph{defined} as homotopy 0-types, without forcing the appearance of any exact completion.
Semantically, this means working with the internal language of an $(\infty,1)$-topos, within which sits the internal language of a 1-topos.
This would be my personal preferred approach; I will comment on it further in \cref{rmk:univalence}.

\subsection*{Background theory}
\label{sec:background-theory}

We work in an informal constructive set theory, assuming neither the axiom of choice nor the law of excluded middle, with one universe to define a size boundary between large and small categories.
Most or all of our results could probably be formalized in the internal language of an elementary topos containing a universe~\cite{streicher:universes}; or in a membership-based set theory like IZF with a universe (or a weaker variant, since we probably do not need much replacement or collection); or in a dependent type theory with UIP, function extensionality, and quotients, like XTT~\cite{sag:xtt}.
The arguments should be predicative, as long as we allow \iProp, like \iSet, to be a large category.
Importantly, however, we do require effective quotients, so that our category \iSet of sets is exact.

\subsection*{Acknowledgments}
\label{sec:acknowledgments}

I would like to thank Peter LeFanu Lumsdaine, Christian Sattler, Andrew Swan, Simon Henry, Ivan De Liberti, David Roberts, Ulrik Buchholtz, Jacques Carette, and other participants at the Bohemian Logico-Philosophical Caf\'e and the Category Theory Community Server for enlightening discussions.
I am particularly grateful to Ian Coley for a careful reading and helpful feedback.

\section{The free exact completion}
\label{sec:setoids}

We start by reviewing the free exact completion.
Recall that an \textbf{exact category} (in the sense of Barr) is a category with finite limits and such that every internal equivalence relation has a pullback-stable quotient of which it is the kernel.

Let \sE be a 1-category with finite limits; we recall from~\cite{cm:ex-lex} how to build an exact category \eex from it freely.\footnote{$\eex$ is sometimes written $\sE_{\mathsf{ex/lex}}$, to emphasize that we started from a category \sE with only finite limits (i.e.\ one that is left exact, or ``lex'').
  This is to distinguish it from other exact completions such as $\sE_{\mathsf{ex/reg}}$, which requires \sE to be a regular category, and unlike the ex/lex completion is an idempotent operation.}
A first thought might be to take the equivalence relations in \sE as the objects of \eex, each such standing in for the quotient of itself.
This produces a category in which every equivalence relation coming from \sE has an effective quotient (see \cref{thm:ereg}), but it also introduces new equivalence relations that do not yet have quotients.
Thus, we need something more general, which turns out to be the following.

\begin{defn}
  A \textbf{pseudo-equivalence relation} in \sE consists of:
  \begin{itemize}
  \item Objects $X_0$ and $X_1$, with morphisms $s,t:X_1\toto X_0$.
  \item A morphism $r:X_0\to X_1$ such that $sr =tr = 1$.
  \item A morphism $v:X_1\to X_1$ such that $sv = t$ and $tv =s$. 
  \item A morphism $m:X_1 \prescript{t}{}{\times}_{X_0}^s X_1 \to X_1$ such that $sm = s\pi_1$ and $tm = t\pi_2$.
  \end{itemize}
\end{defn}

In other words, a pseudo-equivalence relation has the operations of an internal groupoid, but without any axioms.
In particular, any object $X\in \sE$ induces a ``discrete'' pseudo-equivalence relation with $X_1=X_0=X$; this provides a functor $\sE\into \eex$ to the category \eex defined as follows:

\begin{defn}
  The \textbf{free exact completion} $\eex$ of \sE has:
  \begin{itemize}
  \item As objects, pseudo-equivalence relations.
  \item As morphisms $X\to Y$, equivalence classes of pairs of morphisms $f_0:X_0\to Y_0$ and $f_1:X_1\to Y_1$ in \sE with $sf_1 = f_0s$ and $tf_1 = f_0t$, modulo the relation that $(f_0,f_1)\sim (g_0,g_1)$ if there exists a morphism $h:X_0 \to Y_1$ with $sh=f_0$ and $th=g_0$.
  \end{itemize}
  We refer to a pair $(f_0,f_1)$ as a \textbf{morphism representative}, and an $h$ as a \textbf{witness of equality} of two such.
\end{defn}

\begin{rmk}\label{rmk:bicat}
  A pseudo-equivalence relation can also be defined as an internal \emph{bicategory} in \sE such that any two parallel 1-cells are related by a unique 2-cell and all 1-cells are equivalences.
  The tricategory of such ``locally bidiscrete bigroupoids'' is ``locally tridiscrete'', and its homotopy 1-category (obtained by identifying naturally equivalent functors) is \eex.
  Our results about \eex could be obtained by specializing facts about bicategories and tricategories, but we will give concrete proofs instead.
\end{rmk}

It is proven in~\cite{cm:ex-lex} that \eex is an exact category, and that this construction defines a left pseudo-adjoint to the forgetful 2-functor from exact categories to categories with finite limits.
In particular, the inclusion $\sE\into \eex$ preserves finite limits; but even if \sE was already exact, this functor does not in general preserve quotients of equivalence relations.
The only exception is if \sE is exact and satisfies the ``axiom of choice'' that regular epimorphisms are split, in which case the inclusion $\sE\into \eex$ is an equivalence.

We will not repeat the proofs of these facts, but we sketch the following:

\begin{lem}
  \eex has finite limits.
\end{lem}
\begin{proof}
  The terminal object has $T_0=T_1=1$.
  For pullbacks, suppose given a cospan $X \xto{f} Z \xot{g} Y$ in \eex, select representatives $(f_0,f_1)$ and $(g_0,g_1)$ and define
  \begin{alignat*}{2}
    P_0&=(X_0 \times Y_0) \times_{(Z_0\times Z_0)} Z_1
    & \qquad 
    P_1&= (P_0\times P_0) \times_{(X_0\times X_0\times Y_0\times Y_0)} (X_1\times Y_1).\qedhere
  \end{alignat*}
\end{proof}

The particular objects $P_0$ and $P_1$ constructed above depend on the chosen representatives $(f_0,f_1)$ and $(g_0,g_1)$.
Thus, in the absence of the axiom of choice (now meaning the usual axiom of choice in \iSet), \eex does not have a \emph{specified} pullback functor $(\eex)^{\to\ot}\to \eex$, even if \sE has such a functor.
(Although it does have a specified binary \emph{product} functor.)
The situation with infinite diagrams is even worse: without choice we have no way to select representatives for all the morphisms in the diagram simultaneously, so even if \sE is complete, \eex may not be.

\begin{rmk}
  The category of setoids \emph{is} complete and cocomplete if we regard it as an \cE-category, i.e.\ a category enriched over setoids (see e.g.~\cite{agda-categories}).
  Indeed, from the perspective of \cref{rmk:bicat}, the \cE-category of setoids is a tricategory of certain bicategories, so it can be complete even if its homotopy category is not.
  We will not pursue this direction; the point of this paper is to observe that setoids arise unavoidably in homotopy theory \emph{even} if we try our best to remain in the world of ordinary categories.
  See \cref{sec:anafunctors} for further discussion.
\end{rmk}

We can avoid all these problems with limits and colimits by considering a notion of \emph{coherent} diagrams in \eex.

\begin{defn}
  Let $A$ be a small category.
  A \textbf{coherent $A$-diagram} in \eex is:
  \begin{itemize}
  \item For each object $a\in A$, an object $X_a \in \eex$.
  \item For each morphism $\alpha:a\to a'$ in $A$, a morphism representative $X_\al: X_a \to X_{a'}$, consisting of morphisms $X_{\alpha,0}:X_{a,0}\to X_{a',0}$ and $X_{\alpha,1}:X_{a,1}\to X_{a',1}$ in \sE with $s X_{\alpha,1} = X_{\alpha,0} s$ and $t X_{\alpha,1} = X_{\alpha,0} t$.
  \item For each $a\in A$, a morphism $X_r : X_{a,0}\to X_{a,1}$ with $s X_r = 1$ and $t X_r = X_{1_a,0}$ (i.e.\ a witness that $X_{1_a} \sim 1$).
  \item For each $\alpha:a\to a'$ and $\alpha':a'\to a''$, a morphism $X_{\alpha,\alpha'}:X_{a,0} \to X_{a'',1}$ with $s X_{\al,\al'} = X_{\alpha',0} X_{\al,0}$ and $t X_{\al,\al'} = X_{\al'\al,0}$ (i.e.\ a witness that $X_{\al'} X_\al \sim X_{\al'\al}$).
  \end{itemize}
  For coherent $A$-diagrams $X$ and $Y$,
  a \textbf{morphism representative} $f:X\to Y$ is:
  \begin{itemize}
  \item For each $a\in A$, morphisms $f_{a,0}:X_{a,0}\to Y_{a,0}$ and $f_{a,1}:X_{a,1}\to Y_{a,1}$ with $sf_{a,1} = f_{a,0}s$ and $tf_{a,1} = f_{a,0}t$ (i.e.\ a representative of a morphism $X_a \to Y_a$).
  \item For each $\alpha:a\to a'$ in $A$, a morphism $f_\alpha : X_{a,0} \to Y_{a',1}$ with $s f_\alpha = Y_{\alpha,0} f_{a,0}$ and $t f_\alpha = f_{a',0} X_{\alpha,0}$ (i.e.\ a witness that $Y_{\al} f_a \sim f_{a'} X_\al$).
  \end{itemize}
  A \textbf{witness of equality} between two such representatives is
  \begin{itemize}
  \item a family of morphisms $h_a:X_{a,0} \to Y_{a,1}$ with $s h_a = f_{a,0}$ and $t h_a = g_{a,0}$.
  \end{itemize}
  The \textbf{morphisms} of coherent diagrams are the equivalence classes of morphism representatives, modulo the existence of a witness of equality.
  This defines \textbf{the category of coherent diagrams}, which we denote $\eex(A)$.
\end{defn}

\begin{lem}
  If $A=\done$ is the terminal category, then $\eex(\done)\simeq \eex$.
\end{lem}
\begin{proof}
  This is not a definitional equality, since an object of $\eex(\done)$ contains the additional data of an endomorphism representative with witnesses that it is idempotent and equal to the identity.
  But it is straightforward to see that these additional data are redundant.
\end{proof}

\begin{rmk}\label{rmk:nonassoc}
  The 1-category \eex can be expressed as the hom-wise quotient of a 1-category of pseudo-equivalence relations and morphism representatives, as studied in~\cite{kp:cat-setoid}.
  But the same is not true of $\eex(A)$: its morphism representatives cannot be composed associatively (though they become associative after quotienting by witnesses of equality).
  From the perspective of \cref{rmk:bicat}, $\eex(A)$ is the homotopy 1-category of a tricategory of trifunctors. 
\end{rmk}

\begin{rmk}\label{rmk:ac-dia}
  If the axiom of choice holds, then because the equivalence relation on morphisms in $\eex(A)$ makes no reference to $f_{a,1}$ or $f_\al$, instead of including the latter as data in a morphism we can simply assert that for each $a$ or $\al$ such a morphism exists.
  Similarly, since the definition of morphisms makes no reference to $X_r$ or $X_{\al,\al'}$, up to equivalence of categories we can simply assert that these exist.
  The latter assertion then says simply that $X$ is a functor $A\to \eex$, and similarly the former says that morphism is just a natural transformation.
  Thus, the axiom of choice implies that $\eex(A) \simeq (\eex)^A$.
  Note that this is the axiom of choice for the ambient set theory, not the ``axiom of choice'' that regular epimorphisms split in \sE (though of course the two coincide if $\sE=\iSet$).
  In addition, even in the absence of the axiom of choice this holds whenever $A$ is a finite category. 
\end{rmk}

\begin{eg}
  If $u:A\to B$ is a functor between small categories and $X\in \eex(B)$, we have a coherent diagram $u^*X\in \eex(A)$ defined by precomposing all the data of $X$ with the action of $u$ on objects and morphisms.
  This defines a \textbf{restriction} functor $u^*:\eex(B) \to \eex(A)$.
  In particular, the functor $p_A : A\to \done$ induces for any $X\in \eex  \simeq \eex(\done)$ a \textbf{constant} coherent diagram $p_A^*X \in \eex(A)$.
\end{eg}

\begin{thm}\label{thm:lim}
  Suppose \sE is complete, with specified limit functors $\sE^A \to \sE$ for all small categories $A$.
  Then each functor $p_A^* : \eex \to \eex(A)$ has a right adjoint.
\end{thm}
\begin{proof}
  We define the ``limit'' of a coherent diagram $Y\in \eex(A)$ as follows.
  Let $L_0$ be the equalizer of the following parallel pair in \sE:
  \[
    \begin{tikzcd}
      \displaystyle\left(\prod_{a\in A} Y_{a,0} \times \prod_{\alpha:a\to a'} Y_{a',1}\right) \ar[r,shift left] \ar[r,shift right] &
      \displaystyle\prod_{\al:a\to a'} (Y_{a',0} \times Y_{a',0}).
  \end{tikzcd}
  \]
  Here the components of the first morphism at $\al:a\to a'$ are $Y_{\al,0}:Y_{a,0}\to Y_{a',0}$ and $1_{Y_{a',0}}$, while those of the second morphism are $s:Y_{a',1} \to Y_{a',0}$ and $t:Y_{a',1} \to Y_{a',0}$.
  Then let $L_1$ be the pullback
  \[ (L_0\times L_0) \;\bigtimes_{\prod_{a\in A} (Y_{a,0} \times Y_{a,0})}\;\textstyle\prod_{a\in A} Y_{a,1}.\]
  Note that $Y$ contains all the necessary data to define these objects, without any choices necessary.
  It is straightforward to show that $L$ is a pseudo-equivalence relation.

  Now we define a counit $p_A^* L \to Y$.
  For each $a$, the components $L_{0} \to Y_{a,0}$ and $L_1 \to Y_{a,1}$ are just the evident projections; and likewise for the morphisms $L_0 \to Y_{a',1}$ for each $\al:a\to a'$.

  It remains to show that any morphism $f:p_A^* X \to Y$ factors uniquely through $L$.
  Choose a representative of $f$; then the components $f_{a,0} : X_0 \to Y_{a,0}$ and $f_{a,1}:X_1 \to Y_{a,1}$ and $f_{\al}:X_0 \to Y_{a',1}$ are exactly what is needed to define morphisms $\fbar_0:X_0\to L_0$ and $\fbar_1:X_1\to L_1$ with $s\fbar_1 = \fbar_0s$ and $t\fbar_1 = \fbar_0t$.
  Moreover, the representatives of the composite $p_A^*X \to p_A^*L \to Y$ are literally equal in \sE to those of $f$, so we can choose $h_a = r f_{a,0}$ to exhibit this composite as equal to $f$ in $\eex(A)$.

  Finally, suppose we have $g:X\to L$ is such that the composite $p_A^*X \xto{g} p_A^*L \to Y$ is equal to $f$ in $\eex(A)$.
  Choosing a representative for $g$, we obtain components $g_{a,0} : X_0 \to Y_{a,0}$ and $g_\al : X_0 \to Y_{a',1}$ and $g_{a,1}:X_1 \to Y_{a,1}$ satisfying the appropriate equations.
  Choosing a witness of equality to $f$, we have morphisms $h_a : X_a \to Y_{a,1}$ with $sh_a = f_{a,0}$ and $th_a = g_{a,0}$.
  But this is exactly what we need to define a witness $h:X_0 \to L_1$ exhibiting $\fbar\sim g$ in \eex.
\end{proof}

For the case of colimits, we need \sE to admit certain free constructions.
Since our eventual interest is mainly in the case $\sE =\iSet$, we will not worry about the minimum this requires of \sE, instead merely noting:

\begin{lem}\label{thm:free-pseqr}
  Suppose \sE has finite limits, and countable coproducts preserved by pullback.
  Then for any parallel pair $R \toto X_0$, there is a pseudo-equivalence relation $X_1\toto X_0$ with a map $\eta : R\to X_1$ over $X_0\times X_0$, such that for any pseudo-equivalence relation $Y_1 \toto Y_0$ and morphism $f_0:X_0 \to Y_0$ with $g:R\to Y_1$ over $f_0\times f_0$, there exists a $f_1 : X_1 \to Y_1$ over $f_0\times f_0$ such that $f_1 \eta = g$:
  \[
    \begin{tikzcd}
      R \ar[ddr,shift left] \ar[ddr,shift right] \ar[drr] \ar[dr] \\
      & X_1 \ar[d,shift left] \ar[d,shift right] \ar[r,dashed,"\exists"'] & Y_1\ar[d,shift left] \ar[d,shift right] \\
      & X_0 \ar[r] & Y_0.
    \end{tikzcd}
  \]
\end{lem}
\begin{proof}
  Define
  \[X_1 = \sum_{n\in \dN\atop \ep_1,\dots,\ep_n \in \{+1,-1\}} R^{\ep_1}\times_{X_0} R^{\ep_2} \times_{X_0} \cdots \times_{X_0} R^{\ep_n}
  \]
  where $R^{+1}$ means the given span $X_0 \ot R \to X_0$ and $R^{-1}$ means the reversed span.
  (The summand for the case $n=0$ is just $X_0$.)
  In the internal language of \sE, $X_1$ is the object of zigzags such as
  \[ x_0 \xto{r_1} x_1 \xot{r_2} x_2 \xot{r_3} \cdots \xto{r_n} x_n \]
  in which each arrow is labeled by an element of $R$, with the two maps $R\toto X_0$ regarded as source and target, and each arrow in the zigzag can point in either direction.
  The resulting $X_1\toto X_0$ is actually the free internal $\dag$-category on the directed graph $R\toto X_0$.

  Finally, given $f_0$ and $g$ as in the statement, we define $f_1$ on each summand of $X_1$ by applying $g$ to each factor of $R$, then the symmetry operation of $Y$ to each factor with $\ep_k=-1$, and then some bracketing of the transitivity operation of $Y$ to combine all the factors into one (in the case $n=0$ this means the reflexivity operation of $Y$).
  The inclusion $\eta$ is the summand with $n=1$ and $\ep_1=+1$, where no operations are needed other than $g$, so we have $f_1\eta=g$.
\end{proof}

We refer to $X_1\toto X_0$ as in \cref{thm:free-pseqr} as the \textbf{free pseudo-equivalence relation} generated by $R\toto X_0$, although to be precise it is only ``weakly free'' (the morphism $f_1$ is not unique).

\begin{thm}\label{thm:colim}
  If \sE has finite limits and small coproducts preserved by pullback, then each functor $p_A^* : \eex \to \eex(A)$ has a left adjoint.
\end{thm}
Note that although we only require \sE to have co\emph{products}, here $A$ is an arbitrary small category; thus \eex has more ``colimits'' (in this sense) than \sE does.
\begin{proof}
  Given $X\in \eex(A)$, let $C_0$ be the coproduct $\sum_{a\in A} X_{a,0}$, and let $C_1$ be the pseudo-equivalence relation on $C_0$ freely generated (as in \cref{thm:free-pseqr}) by
  \[
    \begin{tikzcd}
      \displaystyle\sum_{\al:a\to a'} \left(X_{a,0} \times_{X_{a',0}} X_{a',1}\right)
      \ar[r,shift left] \ar[r,shift right] &
      C_0.
    \end{tikzcd}
  \]
  Here the pullback is the ``object of triples $(x,x',\xi)$'' where $x\in X_{a,0}$, $x'\in X_{a',0}$, and $\xi\in X_{a',1}$ is a witness that $X_{\al,0}(x)\sim x'$.
  The projection to $C_0$ picks out $x$ and $x'$ in the summands $X_{a,0}$ and $X_{a',0}$.
  (Note that neither of these is the copy of $X_{a',0}$ that we pull back over; that is $X_{\al,0}(x)$.)

  Now we define a unit $X \to p_A^* C$.
  For each $a$, the component $X_{a,0} \to C_0$ is just the coproduct inclusion.
  To define the component $X_{a,1} \to C_1$, the idea is to send a witness $\xi \in X_{a,1}$ that $x\sim x'$ to the image under $\eta$ of the witness that $X_{1_a,0}(x) \sim x \sim x'$ obtained by transitivity from $\xi$ and $X_r$.
  And to define the witness $X_{a,0} \to C_1$ of naturality associated to $\al:a\to a'$, the idea is to send $x\in X_{a,0}$ to (the image under $\eta$ of) the reflexivity witness that $X_{\al,0}(x) \sim X_{\al,0}(x)$.

  It remains to show that any morphism $f:X\to p_A^*Y$ factors uniquely through $C$.
  Choose a representative of $f$; then the components $f_{a,0}:X_{a,0} \to Y_0$ define a morphism $C_0 \to Y_0$, while the components $f_{a,1}:X_{a,1}\to Y_1$ and $f_{\al} : X_{a,0} \to Y_1$ can be combined with transitivity, and the freeness of $C$, to induce a morphism $\fbar:C\to Y$.
  The composite components $X_{a,0}\to (p_A^*C)_{a,0} \to (p_A^*Y)_{a,0} = Y_0$ are then literally equal to $f_{a,0}:X_{a,0} \to Y_0$, so we can use $h_a = r f_{a,0}$ to exhibit this composite as equal to $f$ in $\eex(A)$.

  Finally, suppose we have $g:C\to Y$ such that the composite $X \to p_A^*C \xto{g} p_A^*Y$ is equal to $f$ in $\eex(A)$.
  Choosing a representative for $g$, we obtain components $g_{a,0} : X_{a,0} \to Y_{0}$ and $g_\al : X_{a,0} \to Y_{1}$ and $g_{a,1}:X_{a,1} \to Y_{1}$ satisfying the appropriate equations.
  Choosing a witness of equality to $f$, we have morphisms $h_a : X_{a,0} \to Y_{1}$ with $sh_a = f_{a,0}$ and $th_a = g_{a,0}$.
  But this is exactly what we need to define a witness $h:C_0 \to Y_1$ exhibiting $\fbar\sim g$ in \eex.
\end{proof}

Thus, although $\eex$ does not have infinite limits or colimits, or specified pullbacks, there is nevertheless a sense in which it is strongly complete and cocomplete.
In \cref{sec:derex} we will see that derivators give us a way of making this precise.

\begin{rmk}
  Combining \cref{rmk:ac-dia,thm:lim}, we see that if the axiom of choice holds and \sE is complete, then so is \eex (as an ordinary category).
  This was already observed by~\cite{ht:free-regex}; in their construction, the axiom of choice enters in the fact that epimorphisms of presheaves are closed under arbitrary products.

  Similarly, combining \cref{rmk:ac-dia,thm:colim}, we see that if the axiom of choice holds and \sE has small coproducts preserved by pullback, then \eex is cocomplete.
  Related facts were observed by~\cite{menni:thesis} and~\cite{cv:reg-exact-cplt}; the axiom of choice is hidden because they deal explicitly only with finite coproducts.
\end{rmk}

\section{Derivators}
\label{sec:der}

A derivator is an abstraction of the structure possessed by the homotopy categories of diagrams in a complete and cocomplete $(\infty,1)$-category.
Early authors such as~\cite{heller:htpythys,grothendieck:derivateurs,franke:triang} chose slightly different sets of axioms, but nowadays the community seems to have mostly settled on the definition of Grothendieck.
As is often the case, we have to rephrase the definition to make it constructively useful.
We will also follow~\cite{heller:htpythys,coley:half-deriv} in distinguishing \emph{left} and \emph{right} derivators that have only ``colimits'' and ``limits'', respectively.

Let \cCat and \cCAT be the 2-categories of small and large categories.
For $A\in\cCat$, let $\ob{A}$ denote the discrete category on its objects, with inclusion $\incl{A}:\ob{A} \to A$.

\begin{defn}
  A \textbf{prederivator} is a 2-functor $\D:\cCat\op\to\cCAT$.
  A prederivator is a \textbf{semiderivator} if:
  \begin{itemize}[leftmargin=4em]
  \item[(Der1)] $\D: \cCat\op\to\cCAT$ preserves products indexed by projective\footnote{A set $I$ is projective if every surjection $J\epi I$ has a section.  Thus finite sets are always projective, and the axiom of choice is equivalently ``all sets are projective''.} sets.  That is, if $I$ is projective, the functor $\D(\sum_{i\in I} A_i) \to \prod_{i\in I} \D(A_i)$ is an equivalence, in the constructive sense that we have a specified quasi-inverse to it.
  \item[(Der2)] For any $A\in\cCat$, the functor $\incl{A}^* : \D(A) \to \D(\ob{A})$ is conservative (that is, isomorphism-reflecting).
  \end{itemize}
  A \textbf{left derivator} is a semiderivator such that
  \begin{itemize}[leftmargin=4em]
  \item[(Der3L)] Each functor $u^*: \D(B) \to\D(A)$ has a specified left adjoint $u_!$.
  \item[(Der4L)] Given functors $u: A\to C$ and $v: B\to C$ in \cCat, let $(u/v)$ denote their comma category, with projections $p:(u/v) \to A$ and $q:(u/v)\to B$.
      If $B$ is a discrete category, then the canonical mate-transformation $q_!\, p^* \to v^* u_!$ is an isomorphism.\label{item:ider4i}
  \end{itemize}
  Dually, a \textbf{right derivator} is a semiderivator such that
  \begin{itemize}[leftmargin=4em]
  \item[(Der3R)] Each functor $u^*: \D(B) \to\D(A)$ has a specified right adjoint $u_*$.
  \item[(Der4R)] Given $u$ and $v$ as in (Der4L), if instead $A$ is a discrete category, then the mate-transformation $u^* v_* \to p_* q^*$ is an isomorphism.\label{item:ider4ii}
  \end{itemize}
  A \textbf{derivator} is a semiderivator that is both a left derivator and a right derivator.
  Finally, a prederivator is \textbf{strong} if
  \begin{itemize}[leftmargin=4em]
  \item[(Der5)] For any $A\in\cCat$, the induced functor $\D(A\times \dtwo) \to \D(A)^\dtwo$ is full and essentially surjective, where $\dtwo=(0\to 1)$ is the interval category.
  \end{itemize}
\end{defn}

We immediately record the most basic class of examples.

\begin{eg}\label{thm:cat-der}
  Let \sC be an ordinary category, and $\sC(A) = \sC^A$ the functor category, with 2-functorial action by restriction.
  This 2-functor preserves \emph{all} products, and (Der2) holds because isomorphisms in functor categories are pointwise, while (Der5) is obvious since the functor in question is an isomorphism.
  Thus \sC defines a strong semiderivator, which we call a \textbf{representable} semiderivator and abusively denote also by \sC.
  
  If \sC is cocomplete, the restriction functors admit left adjoints given by pointwise Kan extensions; thus (Der3L) holds, and (Der4L) asserts that these Kan extensions are pointwise, so \sC is a left derivator.
  Similarly, if \sC is complete, it is a right derivator.
  In particular, \iSet is a derivator.
\end{eg}

\begin{rmk}
  The usual definition, as e.g.\ in~\cite{groth:der,coley:half-deriv}, differs in that:
  \begin{itemize}
  \item Axiom (Der1) is asserted for all products, not just projectively indexed ones.\footnote{Although sometimes \cCat is replaced in the definition by a smaller 2-category, such as the 2-category of finite categories, finite posets, or finite direct categories.
      In this case (Der1) is weakened to refer only to the coproducts that exist therein, such as finite ones.}
  \item Axiom (Der2) asserts that the family of functors $a^*: \D(A) \to \D(\done)$ are jointly conservative, for all objects $a\in A$.
    This is equivalent to (Der2) in the presence of the classical (Der1), since $\ob{A}\cong \sum_{a\in A} \done$.
  \item Axiom (Der4L) requires that $B$ be the terminal category $\done$, and dually for (Der4R).
    However, by~\cite[Prop.\ 1.26]{groth:der}, in the presence of the classical (Der1) and (Der2) this implies that the same statements hold without \emph{any} restriction on $B$ (see \cref{thm:comma-hoex} below), including in particular our (Der4).
  \end{itemize}
  Thus, the substantial difference is the weakening of (Der1), which is only weaker in the absence of the axiom of choice.\footnote{The assertion of (Der1) for all projective sets is admittedly a fairly transparent trick for forcing the definition to collapse to the classical one in the presence of the axiom of choice, only slightly less blatant than starting with ``if the axiom of choice holds, then\dots''.
  Probably more natural constructively would be to assert (Der1) only for \emph{finite} products.} 
  Our weaker version appears to be necessary constructively; for some explanation, see the proof of \cref{thm:der1}.

  Perhaps surprisingly, our definition suffices for most of the theory of derivators; axiom (Der1) is rarely needed, and usually only for finite products.
  Intuitively, while a classical (pre)derivator has an underlying ordinary category $\D(\done)$, one of our (pre)derivators has an underlying \emph{\iSet-indexed category} consisting of the categories $\D(I)$ where $I$ is a discrete category.
  We can then reproduce the usual theory by using indexed categories in place of ordinary ones.
  (Note that a prederivator is, in particular, a \cCat-indexed category.)
\end{rmk}

For instance, (Der3L) implies that any left derivator admits ``colimit'' functors given by $(p_A)_!$ for the functor $p_A :A\to\done$, left adjoint to the ``constant diagram'' functor $(p_A)^*$, and dually for right derivators and limits.
The standard (Der4) axioms then says that the general ``Kan extension'' functors $u_*$ and $u_!$ can be computed in terms of these, by the usual formula~\cite[Theorem X.3.1]{maclane}.
Our (Der4) says the same in ``indexed'' or ``internal'' language, referring not only to ``global elements'' $c:1 \to \ob{C}$ but to arbitrary ``generalized elements'' $v:I\to \ob{C}$, where $I$ is a set.


We now give some examples of how such ``indexed reasoning'' can be used to reproduce some of the basic results about derivators from the cited references.

\begin{defn}
  For a left derivator \D, a square 2-cell in \cCat:
  \[
    \begin{tikzcd}
      A \ar[r,"p"] \ar[d,"q"'] \twocell{dr} & B \ar[d,"u"] \\
      C \ar[r,"v"'] & D
    \end{tikzcd}
  \]
  is \textbf{\D-exact} if the induced map $q_!\, p^* \to v^* u_!$ is an isomorphism in $\D$.
  Dually, if \D is a right derivator, such a square is \D-exact if the map $u^* v_* \to p_* q^*$ is an isomorphism.
  (If \D is a derivator, then these two maps are adjunction conjugates, hence the two conditions are equivalent.)

  A square is \textbf{left} (resp.\ \textbf{right}) \textbf{homotopy exact} if it is \D-exact for all left (resp.\ right) derivators \D, and \textbf{homotopy exact} if it is \D-exact for all derivators \D.
\end{defn}

Note that left and right homotopy exactness are stronger than homotopy exactness, oppositely to how being a derivator is stronger than being a left or right derivator.
The functoriality property of mates (e.g.~\cite{ks:r2cats}) imply that horizontal and vertical pasting preserves (left and right) homotopy exact squares.

Observe that for a set $I$, an $I$-indexed family of small categories $A:I\to \cCat$ can equivalently be regarded as a category $A$ equipped with a functor $A\to I$, where $I$ denotes also the corresponding discrete category.
That is, $\cCat^I \simeq \cCat/I$.
Moreover, if $f,g:A\to B$ are functors between two objects of $\cCat/I$, any natural transformation $f\Rightarrow g$ in \cCat must in fact lie in $\cCat/I$, since $I$ is discrete.
In particular, a morphism in $\cCat/I$ has a left or right adjoint in $\cCat/I$ if and only if it does so in \cCat.

\begin{lem}[cf.~{\cite[Proposition 1.18]{groth:der}}]\label{thm:radj-hoex}
  For a set $I$, let $r:A\to B$ be a right adjoint in $\cCat/I$.
  Then the identity 2-cell is left homotopy exact:
  \[
    \begin{tikzcd}
      A \ar[r,"r"] \ar[d,"ur"'] \twocell{dr} & B \ar[d,"u"] \\
      I \ar[r,equals] & I
    \end{tikzcd}
  \]
\end{lem}
\begin{proof}
  If $\ell$ is the left adjoint of $r$, then the map $(ur)_!\, r^* \to u_!$ is conjugate to $u^* \to \ell^* (ur)^*$, which is an identity since the entire adjunction lies over $I$; hence it is also an isomorphism.
\end{proof}

\begin{lem}[cf.~{\cite[Proposition 1.26]{groth:der}}]\label{thm:comma-hoex}
  Any comma square is left and right homotopy exact:
  \[
    \begin{tikzcd}
      (u/v) \ar[r,"p"] \ar[d,"q"'] \twocell{dr} & A \ar[d,"u"] \\
      B \ar[r,"v"'] & C
    \end{tikzcd}
  \]
\end{lem}
\begin{proof}
  We prove the left case.
  By (Der2) and (Der4), it suffices to prove that the pasted rectangle on the left below is homotopy exact, in which the left-hand square is also a comma:
  \[
    \begin{tikzcd}
      (q/\incl{}) \ar[r] \ar[d] \twocell{dr} & (u/v)  \ar[r,"p"] \ar[d,"q"'] \twocell{dr} & A \ar[d,"u"] \\
      \ob{B} \ar[r,"\incl{}"'] & B \ar[r,"v"'] & C
    \end{tikzcd}
    \qquad=\qquad
    \begin{tikzcd}
      (q/\incl{}) \ar[r] \ar[d] \twocell{dr} & (u/v\incl{})  \ar[r,"p"] \ar[d,"q"'] \twocell{dr} & A \ar[d,"u"] \\
      \ob{B} \ar[r,equals] & \ob{B} \ar[r,"v\incl{}"'] & C
    \end{tikzcd}
  \]
  But this is equal to the pasted rectangle on the right above, where the right-hand square is a comma and the left-hand square is an identity.
  And the induced functor $(q/\incl{}) \to (u/v\incl{})$ is a right adjoint, so by \cref{thm:radj-hoex} and (Der4) both of these squares are homotopy exact.
\end{proof}

\begin{lem}[cf.~{\cite[Proposition 1.24]{groth:der}}]\label{thm:opf-hoex}
  If $u$ is a cloven Grothendieck opfibration, then the identity in any pullback square is left homotopy exact:
  \[
    \begin{tikzcd}
      A \ar[r,"p"] \ar[d,"q"'] \drpullback & B \ar[d,"u"] \\
      C \ar[r,"v"'] & D
    \end{tikzcd}
  \]
  Dually, if $v$ is a cloven Grothendieck fibration, such a pullback square is right homotopy exact.
\end{lem}
\begin{proof}
  We prove the left case.
  Let $\incl{D}^*(B)$ be the pullback
  \[
    \begin{tikzcd}
      \incl{D}^*(B) \ar[d] \ar[r] \drpullback & B \ar[d,"u"] \\
      \ob{D}\ar[r,"\incl{D}"'] & D
    \end{tikzcd}
  \]
  Then there is an induced functor $\incl{D}^*(B) \to (u/\incl{D})$, and the cleaving of $u$ supplies a left adjoint to it over $\ob{D}$.
  Similarly, since $q$ is also a cloven opfibration, the induced functor $\incl{C}^*(A) \to (q/\incl{C})$ is a right adjoint over $\ob{C}$.
  Therefore, by (Der2) and (Der4) and \cref{thm:radj-hoex}, it suffices to prove that the following pasting is homotopy exact:
  \[
    \begin{tikzcd}
      \incl{C}^*(A) \ar[r]\ar[d] \twocell{dr} & (q/\incl{C}) \ar[r] \ar[d] \twocell{dr} & A \ar[r,"p"] \ar[d,"q"'] \drpullback  & B \ar[d,"u"] \\
      \ob{C} \ar[r,equals] & \ob{C} \ar[r,"\incl{C}"'] & C \ar[r,"v"'] & D.
    \end{tikzcd}
  \]
  But this factors as
  \[
    \begin{tikzcd}
      \incl{C}^*(A) \ar[r]\ar[d] \drpullback & \incl{D}^*(B) \ar[r] \ar[d] \twocell{dr} & (u/\incl{D}) \ar[r] \ar[d] \twocell{dr} & B \ar[d,"u"] \\
      \ob{C} \ar[r,"\ob{v}"'] & \ob{D} \ar[r,equals] & \ob{D} \ar[r,"\incl{D}"'] & D.
    \end{tikzcd}
  \]
  Here the left- and right-hand squares are homotopy exact by (Der4), while the middle square is homotopy exact by \cref{thm:radj-hoex}.
\end{proof}


\begin{defn}
  A \textbf{morphism} of prederivators is a pseudonatural transformation, and a \textbf{transformation} is a modification.
  We say a morphism $G:\D\to\D'$ of left derivators is \textbf{cocontinuous} if for any functor $u:A\to B$, the canonical mate-transformation
  \[
    \begin{tikzcd}
      \D(A) \ar[r,"G"] \ar[d,"u_!"'] \twocell{dr} & \D'(A) \ar[d,"u_!"] \\
      \D(B) \ar[r,"G"'] & \D'(B)
    \end{tikzcd}
  \]
  is an isomorphism.
  We denote the category of morphisms and transformations by $\hom(\D,\D')$, and its full subcategory of cocontinuous morphisms by $\cchom(\D,\D')$.\footnote{Sometimes the notation $\hom_!$ is used instead, but I find this insufficiently visually distinctive.}
\end{defn}

\begin{lem}\label{thm:cocts-disc}
  A morphism $G:\D\to\D'$ is cocontinuous if and only if the above condition holds when $B$ is discrete.
\end{lem}
\begin{proof}
  By functoriality of mates, combined with (Der2) and (Der4), we can deduce the condition for arbitrary $u:A\to B$ from the condition for $q:(u/\incl{B}) \to \ob{B}$.
\end{proof}

\begin{thm}[in classical mathematics]\label{thm:spaces-afcp}
  Every Quillen model category \cM induces a derivator $\nHo(\cM)$.
  If $\nsSet$ denotes the Kan--Quillen model category of simplicial sets, then $\iSpace = \nHo(\nsSet)$ is the free cocompletion of a point: there is an object $\ast \in \iSpace(\done)$ such that for any derivator \D, the induced functor
  \[ \cchom(\iSpace,\D) \to \D(\done) \]
  is an equivalence of categories.
\end{thm}
\begin{proof}
  In essense, this is due to Heller~\cite{heller:htpythys} and Cisinski~\cite{cisinski:presheaves,cisinski:basicloc}.
\end{proof}


We will also need two-variable morphisms of derivators, as in~\cite{gps:additivity}.

\begin{lem}[{cf.~\cite[Theorem 3.11]{gps:additivity}}]
  For prederivators $\D_1,\D_2,\D_3$, to give a morphism $\D_1\times\D_2 \to \D_3$ is equivalent to giving a family of functors
  \[ \D_1(A) \times \D_2(B) \to \D_3(A\times B) \]
  varying pseudonaturally over $\cCat\op\times\cCat\op$.\qed
\end{lem}

If $\oast$ is such a two-variable morphism, we write $\oast_A : \D_1(A) \times \D_2(A) \to \D_3(A)$ for its components in the ordinary (or ``internal'') sense, and $\oast : \D_1(A) \times \D_2(B) \to \D_3(A\times B)$ for the above equivalent ``external'' components.
The relationship is that $M \oast_A N \cong \Delta_A^*(M\oast N)$ while $M\oast N \cong \pi_1^* M \oast_{A\times B} \pi_2^* N$.

\begin{defn}
  A morphism $\oast:\D_1\times\D_2 \to \D_3$ of left derivators is \textbf{cocontinuous in its first variable} if for any $u:A\to B$ and $M\in \D_1(A)$ and $N\in \D_2(C)$, the following mate-transformation is an isomorphism in $\D_3(B\times C)$:
  \[ (u\times 1)_!\,(M\oast N) \too (u_!\,M) \oast N. \]
\end{defn}

See~\cite[Warning 3.6]{gps:additivity} for why this has to be formulated with the external product rather than the internal one.
There is a dual notion of cocontinuity in the second variable, and an analogue of \cref{thm:cocts-disc} for two-variable morphisms.

Finally, since $\D(A) \to \D(A)\times \D(A)$ is equivalent to $\nabla^* : \D(A) \to \D(A+A)$ (this uses (Der1) for finite coproducts), in a right derivator the former functor also has a right adjoint.
Thus any right derivator \D is ``cartesian monoidal'', with a product morphism $\times : \D\times \D\to\D$.

\begin{defn}\label{defn:distrib}
  We say a derivator \D is \textbf{distributive} if this $\times$ is cocontinuous in both variables.\footnote{Technically this definition does not require \D to be a full derivator, only a ``left derivator with binary products'', but we will have no use for that generality.}
\end{defn}

  For example, a complete and cocomplete category regarded as a derivator as in \cref{thm:cat-der} is distributive if binary products preserve colimits in each variable, in the usual sense.
  In particular, \iSet is distributive.

\section{The derivator of setoids}
\label{sec:derex}

Let \sE be, to start with, a category with finite limits.

\begin{lem}
  $\eex : \cCat\op\to\cCAT$ is a 2-functor.
\end{lem}
\begin{proof}
  First, the restriction functors $u^*:\eex(B) \to \eex(A)$ are strictly functorial, being given by simple composition with the data of $u$.
  Second, given a natural transformation $\mu : u\Rightarrow v :A\to B$ with components $\mu_a : ua \to va$, for any $X\in \eex(B)$ we have an induced family of morphisms $X_{\mu_a,0}:X_{ua,0} \to X_{va,0}$ and $X_{\mu_a,1}:X_{ua,1} \to X_{va,1}$.
  Third, for $\al:a\to a'$, by applying the pseudo-transitivity $m$ to $X_{u\al,\mu_{a'}}$ and $X_{\mu_a,v\al}$, we have morphisms $X_{ua,0} \to X_{va',1}$ exhibiting naturality.
  Thus, we obtain a morphism of coherent diagrams $u^*X \to v^*X$.
  The 2-functoriality axioms follow straightforwardly.
\end{proof}

\begin{lem}\label{thm:der1}
  \eex satisfies (Der1).
\end{lem}
\begin{proof}
  For any coproduct of categories, the functor $\eex(\sum_i A_i) \to \prod_i \eex(A_i)$ is bijective on objects.
  To show that it is full, we must select representatives for a family of morphisms in each $\eex(A_i)$ to assemble them into a representative for a morphism in $\eex(\sum_i A_i)$; this is possible when $I$ is projective.
  Similarly, to show that it is faithful, we must select witnesses of equality in each $\eex(A_i)$ to assemble into such a witness in $\eex(\sum_i A_i)$, which is also possible when $I$ is projective.
\end{proof}

\begin{lem}\label{thm:der2}
  \eex satisfies (Der2).
\end{lem}
\begin{proof}
  Let $f:X\to Y$ be a (representative of a) morphism in $\eex(A)$, with components $f_{a,0}$, $f_{a,1}$, and $f_\al$.
  If it is invertible in $\eex(\ob{A})$, then we have families of morphisms $g_{a,0}:Y_{a,0} \to X_{a,0}$ and $g_{a,1}:Y_{a,1} \to X_{a,1}$ representing morphisms of pseudo-equivalence relations $Y_a\to X_a$, and such that $g f \sim 1$ and $f g \sim 1$ in $\eex(\ob{A})$.
  The latter mean that there exist $h_a : X_{a,0} \to Y_{a,1}$ with $sh_a = g_{a,0}f_{a,0}$ and $th_a = 1$, and also $k_a : Y_{a,0} \to X_{a,1}$ with $sk_a = f_{a,0} g_{a,0}$ and $tk_a = 1$.
  Using a chosen such $h$ and $k$, we can define (copying the usual proof that a pointwise invertible natural transformation is invertible in the functor category) for each $\al:a\to a'$ a morphism $g_{\al}:Y_{a,0} \to X_{a',1}$ making $g$ a representative of a morphism in $\eex(A)$.
  The same $h$ and $k$ then witness that $g f = 1$ and $f g = 1$ in $\eex(A)$.
\end{proof}

\begin{lem}\label{thm:der3}\label{thm:der4}
  If \sE is complete, then \eex is a right derivator.
  If \sE has pullback-stable coproducts, then \eex is a left derivator.
\end{lem}
\begin{proof}
  We can use the classical construction of pointwise Kan extensions~\cite[Theorem X.3.1]{maclane} essentially verbatim, due to the fact that the constructions in \cref{thm:lim,thm:colim} are not just adjoints, but have a constructive universal property with respect to \emph{representatives} of morphisms and \emph{witnesses} of equality.
  That is, there is a function which, given a representative for a morphism $(p_A)^* X \to Y$ in $\eex(A)$, produces a representative for the corresponding morphism $X \to L$, where $L$ is the limit constructed in \cref{thm:lim}; and similarly for witnesses of equality between morphisms, and for colimits.
  The construction of these functions is essentially contained in the proofs of \cref{thm:lim,thm:colim}.

  Consider the case of limits; the case of colimits is analogous.
  Given $u:A\to B$, for any $b\in B$ we have the comma category $(b/u)$ with projection $q_b:(b/u) \to A$.
  For $X\in \eex(A)$ and $b\in B$, define $(u_*X)_b = (p_{b/u})_* q_b^* X$, with the limit functor $(p_{b/u})_*$ constructed as in \cref{thm:lim}.
  For a morphism $\beta:b\to b'$ in $B$, the above remark implies that we can give a morphism representative $(u_*X)_b \to (u_*X)_{b'}$ by giving a morphism representative $(p_{b'/u})^* (p_{b/u})_* q_b^* X \to q_{b'}^* X$, consisting of morphism representatives $(p_{b/u})_* q_b^* X \to X_{a}$ for all morphisms $\beta':b'\to ua$, with compatibility witnesses.
  These latter representatives can be given by the projections from $(p_{b/u})_* q_b^* X$ corresponding to the composite $\beta'\beta : b\to ua$, and similarly for the compatibility witnesses.
  Likewise, the same principles yield witnesses of functoriality and a universal property of $u_*$ as a right adjoint of $u^*$.
  Thus (Der3R) holds.
  To prove (Der4R), in a comma square with $A$ discrete:
  \[
    \begin{tikzcd}
      (u/v) \ar[r,"p"] \ar[d,"q"'] \twocell{dr} & A \ar[d,"u"] \\
      B \ar[r,"v"'] & C,
    \end{tikzcd}
  \]
  the construction above shows that $(u^* v_* X)_a$ and $(p_* q^* X)_a$ are limits (as in \cref{thm:lim}) of the restrictions of $X$ to a pair of isomorphic categories $(ua/v)$ and $(a/p)$.
  Thus, these limits are isomorphic, in a constructive way that can be done simultaneously for all $a\in A$.
\end{proof}

\begin{lem}\label{thm:der5}
  \eex satisfies (Der5).
\end{lem}
\begin{proof}
  Analogously to \cref{rmk:ac-dia}, since \dtwo is finite, the functor in question is actually an equivalence.
\end{proof}

\begin{thm}\label{thm:eex-der}
  For any complete category \sE with small coproducts preserved by pullback, \eex is a strong distributive derivator.
\end{thm}
\begin{proof}
  We have verified all the strong derivator axioms in \cref{thm:der1,thm:der2,thm:der3,thm:der4,thm:der5}, so it remains only to prove distributivity.
  For this, we note that if in \cref{thm:lim} $A$ is discrete, we can replace the construction given there by the simpler $L_0 = \prod_a Y_{a,0}$ and $L_1 = \prod_a Y_{a,1}$.
  Now since the ``colimits'' in \cref{thm:colim} are constructed out of pullbacks and coproducts, and both of these are preserved in each variable by finite products, it follows that the derivator products in \eex preserve its left Kan extensions in each variable.
\end{proof}

\begin{cor}
  \sex is a strong distributive derivator.\qed
\end{cor}

\begin{rmk}
  The free exact completion is not in general idempotent.
  In particular, we can have $(\sex)_{\mathsf{ex}} \not\simeq \sex$.
  However, since $\sex$ is not complete or cocomplete as a \emph{category}, \cref{thm:eex-der} does not imply that $(\sex)_{\mathsf{ex}}$ is a derivator.
  It is unclear whether there is a notion of ``exact completion of a derivator''.
\end{rmk}

\section{Equivalences and locality}
\label{sec:deqv}

As suggested in the introduction, we are interested in derivators that satisfy a relative version of \cref{thm:spaces-afcp}, being a free cocompletion of a point in a world of ``1-categorical derivators''.
Thus, we may start by asking what it is that makes a derivator 1-categorical.
Intuitively, an $(\infty,1)$ ``is'' a 1-category if all its hom-spaces are 0-truncated; but a derivator does not have explicit hom-spaces.

However, we can detect the same information using limits and colimits of constant diagrams.
For instance, for any object $M$ of an $(\infty,1)$-category, the limit of the constant diagram
\[ M \toto M \]
is the \emph{free loop space object} $LM$ of $M$, which is equivalent to $M$ just when $M$ is 0-truncated.
Similarly, one dimension down, the product $M\times M$ is equivalent to $M$ just when $M$ is $(-1)$-truncated, i.e.\ subterminal.
Thus, the ``1-categorical'' or ``0-categorical'' nature of an $(\infty,1)$-category is detected by limits of constant diagrams of this shape.

More generally, in any derivator we can consider the following relative notion.

\begin{defn}\label{defn:deqv}
  Let $u:A\to B$ and $v:B\to I$ be functors, where $I$ is a discrete set.
  We say $u$ is a \textbf{\D-equivalence over $I$}, for a prederivator \D, if $u^*$ is fully faithful on the image of $v^*$.
\end{defn}

\begin{lem}\label{thm:deqv}
  If \D is a left derivator, then $u$ is a \D-equivalence over $I$ if and only if the map $(vu)_!\,(vu)^* \to v_!\, v^*$ is an isomorphism.
  Dually, if \D is a right derivator, then $u$ is a \D-equivalence over $I$ if and only if the map $v_*\, v^* \to (vu)_*\,(vu)^*$ is an isomorphism.
\end{lem}
\begin{proof}
  By the Yoneda lemma, the stated condition for left derivators is equivalent to saying that
  \[ \D(I)(v_!\, v^* X, Y) \to \D(I)((vu)_!\,(vu)^* X ,Y) \]
  is a bijection for all $X,Y\in\D(I)$.
  But this map is isomorphic to
  \[ \D(I)(v^* X, v^* Y) \to \D(I)(v^* u^* X , v^* u^* Y), \]
  and this being a bijection for all $X,Y$ is \cref{defn:deqv}.
\end{proof}

The above considerations might lead us to say that a prederivator \D is 1-categorical if the functor $(\cdot \toto \cdot) \to \done$ is a \D-equivalence, and 0-categorical if the functor $(\done+\done) \to \done$ is a \D-equivalence.
However, as we will see, things are a bit more subtle than this.
We begin by recording some basic properties of the \D-equivalences.

\begin{lem}\label{thm:deqv-pb}
  If $f:I\to J$ is a function between discrete sets and $u:A\to B$ is a \D-equivalence over $J$, for a left or right derivator \D, then the pullback $f^*(u)$ is a \D-equivalence over $I$:
  \[
    \begin{tikzcd}
      f^*A \ar[r,"f^*u"] \ar[d] \drpullback & f^*B \ar[r,"f^*v"] \ar[d] \drpullback & I \ar[d,"f"]\\
      A \ar[r,"u"'] & B \ar[r,"v"'] & J
    \end{tikzcd}
    \]
\end{lem}
\begin{proof}
  We prove the left case.
  Any functor with discrete codomain is a cloven opfibration, so by \cref{thm:opf-hoex} $f^*$ transforms $v_!$ and $(vu)_!$ into $(f^*v)_!$ and $((f^*v)(f^*u))_!$.
  Since it also commutes with $v^*$ and $(vu)^*$ by functoriality, it preserves the property in \cref{thm:deqv}.
\end{proof}

\begin{lem}\label{thm:deqv-coprod}
  Let $I = \sum_{j\in J} I_j$ be a coproduct of sets, with injections $g_j : I_j \to I$, such that the indexing set $J$ is projective.
  If $u:A\to B$ is a functor over $I$ such that each $g_j^*(u)$ is a \D-equivalence over $I_j$ for a left or right derivator \D, then $u$ is a \D-equivalence over $I$.
\end{lem}
\begin{proof}
  By (Der1), isomorphisms in $\D(I)$ are detected in each $\D(I_j)$, and restriction along $g_j$ commutes with the relevant functors as in \cref{thm:deqv-pb}. 
\end{proof}

\begin{cor}\label{thm:deqv-ac}
  Assuming the axiom of choice, $u:A\to B$ is a \D-equivalence over $I$ if and only if its fiber $u_i : A_i \to B_i$ over each $i\in I$ is a \D-equivalence over \done.\qed
\end{cor}

\cref{thm:deqv-ac} explains why in classical mathematics, \D-equivalences are defined without reference to an indexing set $I$.
Note also that for any $f:I\to J$, a \D-equivalence over $I$ is also a \D-equivalence over $J$.
In particular, any \D-equivalence over $I$ is also a \D-equivalence over \done.
Dually, for any functor $u:A\to B$ there is a strongest sort of \D-equivalence that it can be, namely over the set $I = \pi_0(B)$ of connected components of $B$.

\begin{lem}\label{thm:deqv-sat}
  For any prederivator \D, the \D-equivalences are saturated, in the sense that if a morphism $u$ in $\cCat/I$ becomes an isomorphism in $(\cCat/I)[(\W^\D_I)^{-1}]$, where $\W^\D_I$ denotes the \D-equivalences over $I$, then $u$ is a \D-equivalence.
  Therefore, the \D-equivalences satisfy the 2-out-of-3 property, the 2-out-of-6 property, and are closed under retracts.
\end{lem}
\begin{proof}
  For fixed $X,Y\in \D(I)$, there is a functor $\Phi_{X,Y}:\cCat/I \to \iSet\op$ sending $v:A\to I$ to $\D(A)(v^*X, v^*Y)$.
  Since $\Phi_{X,Y}$ inverts all \D-equivalences, it factors through $(\cCat/I)[(\W^\D_I)^{-1}]$; and therefore, if $u$ becomes an isomorphism in $(\cCat/I)[(\W^\D_I)^{-1}]$, it is inverted by $\Phi_{X,Y}$.
  But if $u$ is inverted by $\Phi_{X,Y}$ for all $X,Y$, then it is a \D-equivalence by definition.
\end{proof}

We now give some examples of \D-equivalences.

\begin{prop}\label{thm:set-eqv}
  For any complete or cocomplete category \sC, regarded as a derivator,
  a functor $u:A\to B$ is a \sC-equivalence over $I$ if:
  \begin{itemize}
  \item For each $i\in I$, the functor on fibers $u_i : A_i\to B_i$ induces a bijection on sets of connected components, $\pi_0(u_i) : \pi_0(A_i) \cong \pi_0(B_i)$.
  \end{itemize}
  The converse holds for $\sC=\iSet$.
\end{prop}
\begin{proof}
  In the cocomplete case, we observe that for $v:B\to I$ where $I$ is discrete, and $X\in \sC^I$, we have $(v_!\, v^* X)_i = \pi_0(B_i)\cdot X_i$, the copower of $X_i$ by the set $\pi_0(B_i)$.
  Thus the map $(vu)_!\,(vu)^* \to v_!\,v^*$ consists of copowers by $\pi_0(u_i)$, so it is an isomorphism if these functions are bijections.
  The converse when $\sC=\iSet$ follows by taking $X_i=1$.
\end{proof}

In particular, the functor $(\cdot \toto \cdot) \to \done$ above is a \sC-equivalence for any such \sC.

\begin{defn}
  If \T and \D are prederivators and every \T-equivalence is a \D-equivalence, we say that \D is \textbf{\T-local}.
\end{defn}

Thus \cref{thm:set-eqv} says that \emph{any complete or cocomplete category \sC is \iSet-local}.
For many such \sC the converse also holds (i.e.\ \iSet is \sC-local), but not all.

\begin{prop}\label{thm:tv-eqv}
  If \sC is a complete lattice, regarded as a derivator, then $u:A\to B$ is a \sC-equivalence over $I$ if:
  \begin{itemize}
  \item For each $i\in I$, if $B_i$ is inhabited then so is $A_i$.
  \end{itemize}
  The converse holds when $\sC=\iProp$ is the poset of truth values.
  Thus, every complete lattice is \iProp-local.
\end{prop}
Put differently, the condition is that $u_i$ induces an isomorphism of \emph{supports} $\pi_{-1}(A_i) \cong \pi_{-1}(B_i)$, where $\pi_{-1}(C)$ is the subterminal set corresponding to the proposition ``$C$ is inhabited''.
\begin{proof}
  For $v:B\to I$ with $I$ discrete, and $X\in \sC^I$, we have $(v_!\, v^* X)_i = \bigvee_{b\in B_i} X_i$, and the join of a constant family (a copower in a lattice) depends only on the support of the indexing set.
  The converse when $\sC=\iProp$ follows by taking $X_i = \top$.
\end{proof}

\begin{rmk}
  A functor $u:A\to B$ is an \iSet-equivalence over $I$ if and only if it is a $\iSet$-equivalence over \done, since $\pi_0(A) \cong \sum_i \pi_0(A_i)$.
  However, this is not the case for \iProp-equivalences.
\end{rmk}

Moving down one more categorical dimension, we have the trivial case:

\begin{prop}
  If $\iContr$ denotes the terminal derivator, every functor is a \iContr-equivalence.\qed
\end{prop}

The subtlety mentioned above is that our derivators of exact completions, though intuitively ``1-categorical'', are nevertheless not \iSet-local.

\begin{prop}\label{thm:sex-eqv}
  Let \sE be a complete category with small coproducts preserved by pullback.
  Then $u:A\to B$ is an \eex-equivalence over $I$ if the following hold:
  \begin{itemize}
  \item There is a function $s:\ob{B} \to \ob{A}$.
  \item There is a function sending any $\beta:b\to b'$ to a zigzag in $A$ from $sb$ to $sb'$
    (and hence similarly for any zigzag in $B$).
  \item There is a function sending each $b\in B$ to a zigzag in $B$ from $b$ to $u s b$.
  \item There is a function sending each $a\in A$ to a zigzag in $A$ from $a$ to $s u a$.
  \end{itemize}
  The converse holds if $\sE=\iSet$.
  Thus, every \eex is \sex-local.
\end{prop}
Note that the existence of the zigzags, plus discreteness of $I$, ensures that $s$ must also be a map over $I$, i.e.\ consist of functions $\ob{(B_i)} \to \ob{(A_i)}$.
\begin{proof}
  Let $u:A\to B$ satisfy the stated conditions and $v:B\to I$ a functor with $I$ discrete.
  Let $X,Y \in\eex(I)$, consisting essentially of an $I$-indexed family of pseudo-equivalence relations.
  We must show that $u^*$ is fully faithful on morphisms between $v^*X$ and $v^*Y$.

  For faithfulness, suppose $f,g:v^*X \to v^*Y$ are morphism representatives and we have a witness of equality consisting of maps $h_a : X_{vua,0}\to Y_{vua,1}$.
  Then $h_{sb}$, for $b\in B$, witness that $f$ and $g$ are equal at objects of the form $usb$.
  And since $v^*X$ and $v^*Y$ both act as the identity on all morphisms of $B$, equality of components of $f$ and $g$ transfers, constructively, across all naturality squares.
  Thus, the assumed zigzags in $B$ can be used to construct a witness that $f\sim g$.

  For fullness, suppose $f:u^*v^*X \to u^*v^*Y$ is a morphism representative.
  Given $b\in B$, we obtain components $g_{b,0} = f_{sb,0}$ and $g_{b,1}=f_{sb,1}$ representing a morphism $g_b:(v^*X)_b \to (v^*Y)_b$.
  For any $\beta:b\to b'$ in $B$, by assumption we have a zigzag from $sb$ to $sb'$; composing naturality squares along this zigzag we can construct a witness $g_\be$ making $g$ a morphism representative $v^*X \to v^*Y$.
  Finally, for any $a\in A$, the assumption yields a zigzag from $a$ to $sua$, which we can use to construct a witness that $u^*(g) \sim f$.

  For the converse, suppose $u:A\to B$ is a \sex-equivalence over $I$, and let $X\in \sex(I)$ be constant at the terminal pseudo-equivalence relation.
  Then by the construction in \cref{thm:colim}, $(v_!\, v^* X)_i$ is the pseudo-equivalence relation on the set $\ob{(B_i)}$ of objects of $B_i$ freely generated by reflexivities and the arrows of $B_i$.
  Thus, its relations are essentially bracketed zigzags in $B_i$.
  Similarly, $((vu)_!\, (vu)^* X)_i$ is the set $\ob{(A_i)}$ with relations being bracketed zigzags in $A_i$.
  The stated conditions are then (modulo adding and removing brackets, which is trivial) precisely what it means for these induced maps to be an isomorphism in $\sex(I)$.
\end{proof}

Note that the conditions in \cref{thm:sex-eqv} are \emph{stronger} than those in \cref{thm:set-eqv}.
Thus \iSet is \sex-local, but \sex is not \iSet-local.
Moreover, in the absence of choice, this inequality is strict;

\begin{prop}
  \sex is \iSet-local if and only if the axiom of choice holds.
\end{prop}
\begin{proof}
  Let $p:E\to B$ be a surjection of sets.
  Regard $B$ as a discrete groupoid, and make $E$ a groupoid such that $p$ is fully faithful (i.e.\ equip it with the kernel pair of $p$, regarded as an equivalence relation).
  Then $\pi_0(E) \cong \pi_0(B) = B$, so $p$ is an \iSet-equivalence.
  But if it is a \sex-equivalence, then $p$ is split.
\end{proof}

However, the functor $(\cdot \toto \cdot) \to \done$ is a \sex-equivalence, so \sex is still intuitively ``1-categorical''.
Two more examples will help to clarify the situation.

\begin{eg}\label{thm:epos}
  Let \sE be a category with small products and coproducts.
  For a small category $A$, let $\epos(A)$ denote the following category:
  \begin{itemize}
  \item An object consists of an object $X_a\in \sE$ for all $a\in A$, together with a morphism $X_\al : X_a\to X_{a'}$ for all $\al:a\to a'$ in $A$.
  \item A morphism representative $f:X\to Y$ consists of a morphism $f_a : X_a\to Y_a$ for all $a\in A$.
    Any two morphism representatives are equivalent.
  \end{itemize}
  Thus $\epos(A)$ is a (large) preorder, and in particular $\epos(\done)$ is (equivalent to) the preorder reflection of $\sE$.

  Arguments like those of \cref{thm:der1,thm:der2}, but simpler, show that $\epos$ satisfies (Der1) and (Der2).
  The constant diagram functor $(p_A)^*: \epos(\done) \to \epos(A)$ has a right and left adjoint given by taking products and coproducts respectively.
  We can then use these to construct pointwise Kan extensions as in \cref{thm:der3,thm:der4}, showing that $\epos$ is a derivator.
  If binary products in \sE preserve coproducts in each variable, then \epos is a distributive derivator.
\end{eg}

\begin{prop}\label{thm:epos-eqv}
  For \sE a category with small products and coproducts, a functor $u:A\to B$ is an $\epos$-equivalence over $I$ if:
  \begin{itemize}
  \item There is a function $\ob{B}\to \ob{A}$ over $I$.
  \end{itemize}
  The converse holds if $\sE=\iSet$.
\end{prop}
\begin{proof}
  For $X\in \epos(I)$, by construction $(v_!\,v^* X)_i$ is the copower $\ob{(B_i)}\cdot X_i$, and similarly $((vu)_!\,(vu)^* X)_i=\ob{(A_i)}\cdot X_i$.
  Thus, the condition given yields a map backwards, hence an isomorphism in $\epos(I)$.
  The converse follows by letting $X_i$ be the terminal object.
\end{proof}

As with the relationship between \iSet and \sex, the condition of \cref{thm:epos-eqv} is stronger than that of \cref{thm:tv-eqv}.
Thus \iProp is \spos-local, but \spos is not \iProp-local.
Indeed, \spos is not even \iSet-local, though it is still ``0-categorical'' in that the functor $(\done+\done)\to\done$ is a \spos-equivalence.

It is true that \spos is \sex-local.
It is also local for the following intermediate derivator \sreg:

\begin{eg}\label{thm:ereg}
  For a category \sE with finite limits, its reg/lex completion \ereg is defined to be the full subcategory of \eex on the pseudo-equivalence relations that are kernel pairs.
  Such kernel pairs are, in particular, actual equivalence relations; and if \sE is already exact (like \iSet), then they include all the equivalence relations.

  If we define $\ereg(A)$ as a similar subcategory of $\eex(A)$, then it is closed under the limits of \cref{thm:lim} but not the colimits of \cref{thm:colim}.
  However, the (regular epi, mono) factorization of a pseudo-equivalence relation always yields an equivalence relation.
  Thus, if \sE is exact, then $\ereg(A)$ is reflective in $\eex(A)$; so we can define left Kan extensions in $\ereg(A)$ by composing the reflection with those of $\eex(A)$.
  Since the reflections commute with the restriction functors, (Der4) holds.

  In sum, if \sE is complete, exact, and has small coproducts preserved by pullback, then $\ereg(A)$ is a derivator.
  Since products preserve image factorizations, $\ereg(A)$ is also a distributive derivator.
\end{eg}

\begin{rmk}
  Analogously to \cref{rmk:bicat}, we can view \sreg as the homotopy category of the bicategory of ``bidiscrete groupoids'' (those in which any two parallel arrows are equal).
  See~\cite{kp:cat-setoid,kinoshita:bicat-ecat}.
\end{rmk}

\begin{prop}\label{thm:sreg-eqv}
  Let \sE be complete, exact, and have small coproducts preserved by pullback.
  Then $u:A\to B$ is an \ereg-equivalence over $I$ if the following hold:
  \begin{itemize}
  \item There is a function $s:\ob{B} \to \ob{A}$.
  \item For any $\beta:b\to b'$ in $B$, there exists a zigzag in $A$ from $sb$ to $sb'$ (and hence likewise for any zigzag in $B$).
  \item For any $b\in B$, there exists a zigzag in $B$ from $b$ to $u s b$.
  \item For any $a\in A$, there exists a zigzag in $A$ from $a$ to $s u a$.
  \end{itemize}
  The converse holds if $\sE=\iSet$.
  Thus, every \ereg is \sreg-local.
\end{prop}
\begin{proof}
  In \ereg, witnesses of equality are unique when they exist; thus it suffices to assert that they exist rather than specifying them functionally.
  Hence, we can essentially copy the proof of \cref{thm:sex-eqv}, but without specified zigzags.
\end{proof}

Clearly \sreg is \sex-local while \iSet is \sreg-local.
Also, \spos is \sreg-local.
Thus, in the preorder where $\D_1 \le \D_2$ means ``$\D_1$ is $\D_2$-local'', we have the fragment shown in \cref{fig:locality}.
In \cref{sec:anafunctors} we will speculate about extending this upwards.

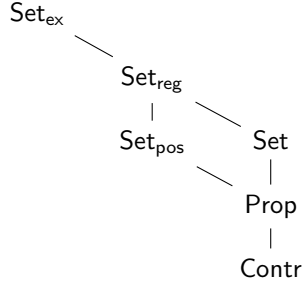
\begin{figure}
  \centering
  \[
  \begin{tikzcd}[row sep={.85cm,between origins},column sep=small]
    \sex \ar[dr,-]\\
    &\sreg \ar[dr,-] \ar[d,-]\\
    &\spos\ar[dr,-] & \iSet \ar[d,-]\\
    && \iProp \ar[d,-] \\
    && \iContr 
  \end{tikzcd}
\]
  \caption{Part of the preorder of relative free cocompletions of a point}
  \label{fig:locality}
\end{figure}

\section{\sex is a relative free cocompletion}
\label{sec:free-cocompletion}

We will show each of the derivators in \cref{fig:locality} is the free cocompletion of a point in the sub-universe of derivators that are local for it, in the following sense.

\begin{defn}
  A left derivator \T is a \textbf{relative free cocompletion of a point} if for any \emph{\T-local} left derivator \D, the ``evaluation at the terminal object $\ast \in \T(\done)$'' functor
  \[ \cchom(\T,\D) \to \D(\done) \]
  is an equivalence of categories.
\end{defn}

How do we prove such universal properties?
As observed by~\cite{heller:htpythys}, there is a derivator that can easily be shown to map into any other left derivator, namely the complete and cocomplete category \cCat.
More generally, we have:

\begin{lem}\label{thm:odot}
  For any left derivator \D, there is a morphism $\odot : \cCat \times \D \to \D$.
  Moreover, if $u:E\to F$ is a morphism in $\cCat^A$ such that $\sum_a u_a$ is a \D-equivalence over $\ob{A}$, then $u\odot M : E\odot M \to F\odot M$ is an isomorphism in \D.
\end{lem}
\begin{proof}
  As in~\cite[Theorem 3.11]{gps:additivity}, to define such a two-variable morphism it suffices to give functors $\odot :\cCat(A) \times \D(B) \to \D(A\times B)$ that vary pseudonaturally in $A$ and $B$.
  The components $\odot_A : (\cCat\times\D)(A) = \cCat(A) \times \D(A) \to\D(A)$ of a pseudonatural transformation are then obtained by composing with restriction along the diagonal $A \to A\times A$.

  Given $E\in \cCat^A$, let $p_E:\tint E\to A$ be its Grothendieck construction, which is a split opfibration.
  Then we have the following diagram:
  \[
    \begin{tikzcd}
      \tint E \times B \ar[r,"\pi_2"] \ar[d,"p_E \times 1"'] & B\\
      A\times B
    \end{tikzcd}
  \]
  Therefore, given $M\in \D(B)$, we can define
  \[E\odot M = (p_E\times 1)_!\, (\pi_2)^* (M) \quad \in \D(A\times B).
  \]
  Pseudonaturality is immediate.

  Now suppose $u:E\to F$ is such that $\sum_a u_a$ is a \D-equivalence over $\ob{A}$.
  To show that $u\odot M$ is an isomorphism, by (Der2) it suffices to restrict it to $\ob{A}\times \ob{B}$.
  And since $p_E\times 1$ and $p_F\times 1$ are opfibrations, by \cref{thm:opf-hoex} the following square is exact, along with the analogous one for $F$:
  \[
    \begin{tikzcd}
      (\sum_a E_a) \times \ob{B} \ar[r] \ar[d,"\ob{(p_E)}\times 1"'] \drpullback & \tint E \times B \ar[d,"p_E\times 1"]\\
      \ob{A} \times \ob{B} \ar[r] & A\times B
    \end{tikzcd}
  \]
  Moreover, the restriction of $M\in \D(B)$ to $(\sum_a E_a) \times \ob{B}$ factors through its restriction to $\ob{B}$ and also to $\ob{A}\times \ob{B}$.
  Now the desired statement simply reduces to the fact that $(\sum_a u_a) \times 1_{\ob{B}}$ is a \D-equivalence over $\ob{A}\times \ob{B}$, which follows from the hypothesis and \cref{thm:deqv-pb}.
\end{proof}

Since left extensions in \D commute with each other, $\odot$ is cocontinuous in its second variable.
If it were also cocontinuous in its first variable, defining $E \mapsto E\odot 1$ would give a cocontinuous morphism $\cCat\to\D$.
This is not generally the case, essentially because $\tint E$ is a \emph{oplax} colimit of $E$ rather than a homotopy colimit.
However, we can make it true by ``localizing \cCat'' in a way that forces such oplax colimits to become ``colimits'' in a derivator.

Classically, there is a universal way to do this, using the Thomason model structure~\cite{thomason:cat-model} on \cCat, which is Quillen equivalent to simplicial sets.
This is roughly the approach of~\cite{heller:htpythys,cisinski:presheaves,cisinski:basicloc}.
Model categories for relative free cocompletions of a point can then be obtained by left Bousfield localization.
It would be interesting to see whether this approach can be reproduced constructively, but we will not attempt to do that here.

Instead, since \cref{fig:locality} contains a maximal element \sex, we will just prove explicitly that \sex is a relative free cocompletion of a point, and then deduce the same property for the other derivators in \cref{fig:locality}.
Of course, a more abstract approach will probably be required to extend these results to higher dimensions.

\begin{defn}\label{thm:Xtil-done}
  For $X\in \sex(\done)$, let $\Xtil$ be the category with object set $X_{0} + X_{1}$ and nonidentity arrows $\xi \to s\xi$ and $\xi \to t\xi$ for all $\xi\in X_{1}$.
\end{defn}

Then $\Xtil\odot 1 \in \sex(\done)$ is the set $X_0 + X_1$ with pseudo-equivalence relation freely generated by $\xi \sim s\xi$ and $\xi\sim t\xi$.

\begin{lem}
  $\Xtil\odot 1$ is isomorphic to $X$ in \sex.
\end{lem}
\begin{proof}
  In one direction, we have a map $X\to \Xtil \odot 1$ that is the inclusion of the summand $X_0$, and sending a witness $\xi$ that $s\xi\sim t\xi$ to the composite witness $s\xi \sim \xi \sim t\xi$.
  In the other direction, we can act as the identity on $X_0$ and send $\xi\in X_1$ to $s\xi$ (say), with the generating witnesses of equality $\xi \sim s\xi$ sent to the reflexivity witness for $s\xi$, and the generating witnesses $\xi\sim t\xi$ sent to the witness $\xi$ that $s\xi\sim t\xi$.
  The composite on $X_0$ is the identity, while the composite on $\Xtil\odot 1$ is equal to the identity via the witnesses $\xi\sim s\xi$.
\end{proof}

We would like to represent a coherent diagram $X\in \sex(A)$ similarly by an object of $\cCat^A$.
However, since $X$ is only functorial up to witnesses of equality, a naive pointwise construction does not produce a functor (or even a pseudofunctor) $\Xtil: A\to \cCat$.
More importantly, the morphisms in $\sex(A)$ are not natural or even pseudonatural for this construction.
Thus, we need some kind of strictification.

\begin{rmk}
  At this point we could attempt to proceed in roughly the same way that derivators are usually constructed in classical homotopy theory (see e.g.~\cite{cisinski:cat-deriv} or~\cite[Proposition 1.30]{groth:der}), by building some kind of \emph{model category} of setoids and morphism representatives whose homotopy category would be $\sex(\done)$.
  We would then lift this model category to a model structure on \emph{strict} $A$-shaped diagrams and strict natural transformations, whose homotopy category would be equivalent to $\sex(A)$.
  The machinery of Quillen adjunctions would then give an alternative approach to the construction of the derivator $\sex$, and the strictness of the morphisms in the model category would make it easier to lift the construction $\Xtil$ to diagrams.

  The first step of this approach to \sex was achieved in~\cite[\S4.1]{henry:weak-modelcats} with the construction of a \emph{weak model category} of setoids whose homotopy category is $\sex(\done)$.
  However, the lifting of weak model structures to categories of diagrams does not exist in the literature yet.
  Rather than develop this machinery here, I have elected to give an explicit construction, which has the additional advantage of being more accessible to a reader without experience in model category theory.
  But it should be clear that this is only feasible because of the very simple nature of the derivator \sex; more complicated examples require more advanced techniques.
\end{rmk}

\begin{defn}
  For $X\in \sex(A)$, let $\Xtil:A\to\cCat$ be the following functor.
  \begin{itemize}
  \item For $c\in A$, the category $\Xtil_{c}$ has two classes of objects:
    \begin{enumerate}
    \item Triples $(a,\al,x)$ where $\al:a\to c$ and $x\in X_{a,0}$, which can be drawn as:
      \[
        \begin{tikzcd}
          x \ar[d,dotted,-] \\
          a \ar[r,"\al"] & c
        \end{tikzcd}
        \]
    \item Tuples $(a,\al,x,a',\al',x',\xi)$ where $\al:a\to a'$ in $A$ and $x\in X_{a,0}$, while $\al':a'\to c$ in $A$ and $x'\in X_{a',0}$, and $\xi\in X_{a',1}$ satisfies $s\xi = X_{\al,0}(x)$ and $t\xi=x'$, as shown:
      \[
      \begin{tikzcd}[column sep=tiny]
        x \ar[d,dotted,-] \ar[rrr,|->] & {}& {}&
        X_{\al,0}(x) \ar[dr,dotted,-] \ar[rr,-,"\xi"]  && x' \ar[dl,dotted,-]
        \\
        a \ar[rrrr,"\al"] &&&&
        a' \ar[rrrr,"\al'"] &&{}&{}&
        c
      \end{tikzcd}
    \]
    \end{enumerate}
  \item The nonidentity morphisms in $\Xtil_c$ are of the form
    \begin{align*}
      (a,\al,x,a',\al',x',\xi) &\to (a,\al'\al,x) \qquad\text{and}\\
      (a,\al,x,a',\al',x',\xi) &\to (a',\al',x').
    \end{align*}
  \item For $\gm:c\to c'$ in $A$, the functor $\Xtil_\gm : \Xtil_c \to \Xtil_{c'}$ is defined on objects by
    \begin{align*}
      \Xtil_\gm(a,\al,x) &= (a,\gm\al,x)\\
      \Xtil_\gm(a,\al,x,a',\al',x',\xi) &= (a,\al,x,a',\gm\al',x',\xi)
    \end{align*}
  \end{itemize}
  For a morphism representative $f:X\to Y$, let $\ftil:\Xtil\to\Ytil$ be the natural transformation whose component $\ftil_c : \Xtil_c \to \Ytil_c$ is defined on objects by
  \begin{align*}
    \ftil_c(a,\al,x) &= (a,\al,f_{a,0}(x))\\
    \ftil_c(a,\al,x,a',\al',x',\xi) &= (a,\al,f_{a,0}(x),a',\al',f_{a,0}(x'),m(f_\al(x),f_{a,1}(\xi))),
  \end{align*}
  where $m$ is the transitivity operation on equality witnesses in $Y_{a'}$.
\end{defn}

\begin{lem}\label{thm:self-odt}
  For any $X\in\sex(A)$ we have a specified isomorphism $\Xtil \odot \mathord\ast\cong X$, where $\ast\in \sex(\done)$ is the terminal object.
  Similarly, for any morphism representative $f:X\to Y$ we have a specified witness that the evident square commutes:
  \[
    \begin{tikzcd}
      \Xtil \odot \ast \ar[r,"\cong"] \ar[d,"\ftil \odot \ast"'] & X \ar[d,"f"]\\
      \Ytil \odot \ast \ar[r,"\cong"'] & Y
    \end{tikzcd}
  \]
\end{lem}
\begin{proof}
  By definition, $\Xtil \odot \mathord\ast$ is the left extension of the constant diagram at $\ast$ along the functor $p_{\Xtil} : \tint\Xtil \to A$.
  Since this functor is a cloven (indeed, split) opfibration, this extension can be computed using colimits, as in \cref{thm:colim}, over the fibers.
  The fiber over $c\in A$ is the category $\Xtil_c$ as defined above.
  Thus, $(\Xtil \odot \mathord\ast)_c$ has underlying set consisting of the triples $(a,\al,x)$ and tuples $(a,\al,x,a',\al',x',\xi)$, with pseudo-equivalence relation freely generated by witnesses
  $(a,\al'\al,x) \sim (a,\al,x,a',\al',x',\xi)$ and $(a,\al,x,a',\al',x',\xi) \sim (a',\al',x')$.

  In one direction, we define a morphism representative $g:\Xtil \odot \mathord\ast \to X$ by
  \begin{align*}
    g_{c,0}(a,\al,x) &= X_{\al,0}(x)\\
    g_{c,0}(a,\al,x,a',\al',x',\xi) &= X_{\al',0}(x')\\
    g_{c,1}((a,\al'\al,x) \sim (a,\al,x,a',\al',x',\xi)) &= m(X_{\al,\al'}(x),X_{\al',1}(\xi))\\
    g_{c,1}((a,\al,x,a',\al',x',\xi) \sim (a',\al',x')) &= r(X_{\al',0}(x'))\\
    g_{\gm}(a,\al,x) &= X_{\al,gm}(x)\\
    g_{\gm}(a,\al,x,a',\al',x',\xi) &= X_{\al',gm}(x')
  \end{align*}
  (extending to all of $(\Xtil \odot \mathord\ast)_{c,1}$ by freeness).
  In the other direction, we define a morphism representative $h: X \to \Xtil \odot \mathord\ast$ by
  \begin{align*}
    h_{c,0}(x) &= (c,1_c,x)\\
    h_{c,1}(\xi) &= (c,1_c,s\xi,c,1_c,t\xi,m(X_r(s\xi),\xi))\\
    h_\gm(x) &= \big(
    (c,\gm,x) \sim
    (c,\gm,x,c',1_{c'},X_{\gm,0}(x),r(X_{\gm,0}(x))) \sim
    (c',1_{c'},X_{\gm,0}(x))
    \big).
  \end{align*}
  The composite in one direction, $h\circ g$, sends $(a,\al,x)$ to $(c,1_c,X_{\al,0}(x))$, for which we have
  \[ (a,\al,x) \sim
    (a,\al,x,c,1_c,X_{\al,0}(x),r(X_{\al,0}(x))) \sim
    (c,1_c,X_{\al,0}(x)).
  \]
  And it sends $(a,\al,x,a',\al',x',\xi)$ to $(c,1_c,X_{\al',0}(x'))$, for which we have
  \[(a,\al,x,a',\al',x',\xi) \sim (a',\al',x')\]
  together with a zigzag like that above.
  And the composite in the other direction, $g\circ h$, sends $x\in X_c$ to $X_{1_c,0}(x)$, which is identified with $x$ by $X_r(x)$.
  Thus, $g$ and $h$ together represent an isomorphism in $\sex(A)$.

  For the second statement, note that $\ftil\odot 1 : \Xtil\odot\ast \to \Ytil\odot \ast$ sends $(a,\al,x)$ to $(a,\al,f_{a,0}(x))$.
  Thus, the composite $X \to \Xtil\odot \ast \to \Ytil\odot \ast$ and $X\to Y \to \Ytil\odot \ast$ both send $x$ to $(c,1_c,f_{c,0}(x))$.
\end{proof}

We emphasize, however, that the construction $f\mapsto \ftil$ does not define any kind of functor yet.
Specifically, it is only defined on morphism \emph{representatives}, which do not compose associatively, and the composite of two morphisms of the form $\ftil$ may no longer be of that form.
Thus, we need some way to also detect witnesses of equality at the categorical level.
For this we use the following ``path space''.

\begin{defn}
  For $X\in \sex(A)$, let $\wp\Xtil:A\to\cCat$ be the following functor.
  \begin{itemize}
  \item For $c\in A$, the category $\wp\Xtil_c$ hos two classes of objects:
    \begin{enumerate}
    \item Triples $(a,\al,\ze)$ where $\al:a\to c$ and $\ze\in X_{a,1}$.
    \item Tuples $(a,\al,\ze,a',\al',\ze',\xi,\xi')$ where $\al:a\to a'$ and $\ze\in X_{a,1}$, while $\al':a'\to c$ and $\ze'\in X_{a',1}$, and $\xi,\xi'\in X_{a',1}$ satisfy $s\xi = X_{\al,0}(s\ze)$, $t\xi = s\ze'$, $s\xi' = X_{\al,0}(t\ze)$, and $t\xi' = t\ze'$.
    \end{enumerate}
  \item The nonidentity morphisms in $\wp\Xtil_c$ are of the form:
    \begin{align*}
      (a,\al,\ze,a',\al',\ze',\xi,\xi') &\to (a,\al'\al,\ze)\\
      (a,\al,\ze,a',\al',\ze',\xi,\xi') &\to (a',\al',\ze').
    \end{align*}
  \item For $\gm:c\to c'$ in $A$, the functor $\wp\Xtil_\gm : \wp\Xtil_c \to \wp\Xtil_{c'}$ is defined on objects by
    \begin{align*}
      \Xtil_\gm(a,\al,\ze) &= (a,\gm\al,\ze)\\
      \Xtil_\gm(a,\al,\ze,a',\al',\ze',\xi,\xi') &= (a,\al,\ze,a',\gm\al',\ze',\xi,\xi').
    \end{align*}
  \end{itemize}
  There are two natural transformations $\si,\tau:\wp\Xtil\to\Xtil$ defined on objects by
  \begin{align*}
    \si_c(a,\al,\ze) &= (a,\al,s\ze)\\
    \si_c(a,\al,\ze,a',\al',\ze',\xi,\xi') &= (a,\al,s\ze,a',\al',s\ze',\xi)\\
    \tau_c(a,\al,\ze) &= (a,\al,t\ze)\\
    \tau_c(a,\al,\ze,a',\al',\ze',\xi,\xi') &= (a,\al,t\ze,a',\al',t\ze',\xi').
  \end{align*}
  Finally, there is a natural transformation $\rho:\Xtil \to \wp\Xtil$ defined on objects by
  \begin{align*}
    \rho_c(a,\al,x) &= (a,\al,rx)\\
    \rho_c(a,\al,x,a',\al',x',\xi) &= (a,\al,rx,a',\al',rx',\xi,\xi),
  \end{align*}
  where $r$ is the witness of reflexivity in $X$.
\end{defn}

\begin{lem}\label{thm:path}
  We have $\si\rho = \tau\rho = 1_{\Xtil}$, and the functors $\sum_a\rho_a$, $\sum_a\si_a$, and $\sum_a\tau_a$ are \sex-equivalences over $\ob{A}$.
\end{lem}
\begin{proof}
  The first statement is evident.
  For the second, by 2-out-of-3 (\cref{thm:deqv-sat}) it suffices to show $\sum_a\rho_a$ is a \sex-equivalence.
  Since $(\sum_a \si_a) \circ (\sum_a \rho_a) = 1$, it suffices to connect each object of $\wp\Xtil_a$ to its image under $\rho_a\si_a$ with a zigzag.

  First we need a zigzag between $(a,\al,\ze)$ and $(a,\al,rs\ze)$, for which we can use
  \[ (a,\al,rs\ze)
    \leftarrow
    (a,1_a,\ze,a,\al,rs\ze,rs\ze,\ze)
    \to
    (a,\al,\ze).
  \]
  Next we need a zigzag between $(a,\al,\ze,a',\al',\ze',\xi,\xi')$ and $(a,\al,rs\ze,a',\al',rs\ze',\xi,\xi)$, for which we compose the zigzag constructed as above for $(a',\al',\ze')$ with the maps
  \begin{align*}
    (a,\al,rs\ze,a',\al',rs\ze',\xi,\xi) &\to (a',\al',rs\ze')\qquad\text{and}\\
    (a',\al',\ze') &\ot (a,\al,\ze,a',\al',\ze',\xi,\xi').\qedhere
  \end{align*}
\end{proof}

\cref{thm:path} says that $\wp\Xtil$ is a ``path space'' relative to the \sex-equivalences.

\begin{defn}
  For morphisms $\phi,\psi : \Xtil \to \Ytil$ in $\cCat^A$, a \textbf{right homotopy} $\phi\sim\psi$ is a morphism $\theta : \Xtil \to \wp\Ytil$ such that $\si\th = \phi$ and $\tau\th = \psi$.
\end{defn}

\begin{lem}\label{thm:eq-rhtpy}
  If $f,g:X\to Y$ are morphism representatives in $\sex(A)$ and $h:f\sim g$ is a witness of equality, then $\ftil$ and $\gtil$ are right homotopic.
\end{lem}
\begin{proof}
  We define $\htil : \Xtil \to \wp\Ytil$ on objects by
  \( \htil(a,\al,x) = (a,\al,h_a(x))\) and
  \begin{multline*}
    \htil(a,\al,x,a',\al',x',\xi) =\\ (a,\al,h_a(x),a',\al',h_{a'}(x'),m(f_\al(x),f_{a,1}(\xi)),m(g_\al(x),g_{a,1}(\xi))).\qedhere
  \end{multline*}  
\end{proof}

We can now use this path-space to remedy the problems of functoriality.

\begin{lem}\label{thm:comp-rhtpy}
  If $X\xto{f} Y \xto{g} Z$ are morphism representatives in $\sex(A)$, then $\gtil\ftil$ and $\widetilde{gf}$ are right homotopic.
\end{lem}
\begin{proof}
  By definition, we have
  \begin{align*}
    \gtil_c (\ftil_c(a,\al,x)) &= (a,\al,g_{a,0}(f_{a,0}(x)))\\
    \gtil_c (\ftil_c(a,\al,x,a',\al',x',\xi)) &=
    \begin{multlined}[t]
      (a,\al,g_{a,0}(f_{a,0}(x)),a',\al',g_{a,0}(f_{a,0}(x')),\\
      m(g_{\al}(f_{a,0}(x)),g_{a,1}(m(f_\al(x),f_{a,1}(\xi)))))
    \end{multlined}\\
    \widetilde{gf}_c(a,\al,x) &= (a,\al,g_{a,0}(f_{a,0}(x)))\\
    \widetilde{gf}_c(a,\al,x,a',\al',x',\xi) &=
    \begin{multlined}[t]
      (a,\al,g_{a,0}(f_{a,0}(x)),a',\al',g_{a,0}(f_{a,0}(x')),\\
      m((gf)_\al(x),g_{a,1}(f_{a,1}(\xi))))
    \end{multlined}
  \end{align*}
  where $(gf)_\al$ is the composite witness of naturality.
  Now define $\mtil : \Xtil \to \wp\Ztil$ by
  \begin{align*}
    \mtil_c(a,\al,x) &= (a,\al,r(g_{a,0}(f_{a,0}(x))))\\
    \mtil_c(a,\al,x,a',\al',x',\xi) &=
    \begin{multlined}[t]
      (a,\al,r(g_{a,0}(f_{a,0}(x))),a',\al',r(g_{a,0}(f_{a,0}(x'))),\\
      m(g_{\al}(f_{a,0}(x)),g_{a,1}(m(f_\al(x),f_{a,1}(\xi)))),\\
      m((gf)_\al(x),g_{a,1}(f_{a,1}(\xi)))).\qedhere
    \end{multlined}
  \end{align*}
\end{proof}

\begin{lem}\label{thm:id-rhtpy}
  For $X\in \sex(A)$, the morphisms $\widetilde{1_X}$ and $1_{\Xtil}$ are right homotopic.
\end{lem}
\begin{proof}
  We can define $\itil : \Xtil \to \wp\Xtil$ by
  \begin{align*}
    \itil_c(a,\al,x) &= (a,\al,rx)\\
    \itil_c(a,\al,x,a',\al',x',\xi) &= (a,\al,rx,a',\al',rx',m((1_X)_\al(x),\xi),\xi).\qedhere
  \end{align*}
\end{proof}

Now we show that right homotopies are inverted in \sex-local derivators.

\begin{lem}\label{thm:rhtpy-quot}
  Let \D be a \sex-local left derivator.
  For any $X\in \sex(A)$ and $M\in \D(B)$, we have
  \[\si_X \odot 1_M = \tau_X \odot 1_M\]
  as morphisms $\wp\Xtil \odot M \to \Xtil\odot M$ in $\D(A\times B)$.
  Therefore, if $\phi,\psi:\Xtil \to \Ytil$ are right homotopic, then $\phi\odot \ell = \psi\odot \ell$ for any $\ell$.
\end{lem}
\begin{proof}
  By functoriality of $\odot$, we have
  \[ (\si\odot 1_M) \circ (\rho \odot 1_M) = (\tau \odot 1_M) \circ (\rho \odot 1_M). \]
  However, by \cref{thm:path}, $\sum_a \rho_a$ is a \sex-equivalence over $\ob{A}$, and hence also a \D-equivalence since \D is \sex-local.
  Therefore, by \cref{thm:odot}, $\rho\odot 1_M$ is an isomorphism, and thus cancellable.
  So $\si\odot 1_M = \tau\odot 1_M$.

  For the last statement, a right homotopy is a $\th$ with $\si\th = \phi$ and $\tau\th = \psi$.
  Thus, the equation $\si\odot 1_M = \tau\odot 1_M$ implies $\phi\odot \ell = \psi\odot \ell$ by functoriality.
\end{proof}

This implies that $\odot$ descends from \cCat to \sex via $\widetilde{(-)}$.

\begin{defn}
  For $X\in \sex(A)$ and $M\in \D(B)$, define $X\odt M = \Xtil \odot M$.
  Similarly, for $f:X\to Y$ in $\sex(A)$ and $\ell:M\to N$ in $\D(B)$, we choose a representative of $f$ and define $f\odt \ell = \ftil\odot \ell$.
\end{defn}

\begin{prop}\label{thm:odt-components}
  If \D is \sex-local,
  the definition of $f\odt g$ is independent of the choice of representative for $f$, and defines a functor
  \[ \odt: \sex(A) \times \D(B) \to \D(A\times B). \]
\end{prop}
\begin{proof}
  By \cref{thm:eq-rhtpy}, any witness of equality $h:f\sim g$ between two morphism representatives yields a right homotopy $\ftil\sim\gtil$.
  Thus, by \cref{thm:rhtpy-quot}, we have $f\odt \ell = \ftil\odot \ell = \gtil\odot \ell = g\odt \ell$.
  Functoriality on $\sex(A)$ follows similarly from \cref{thm:comp-rhtpy,thm:id-rhtpy}.
\end{proof}

Now we have to show that these functors vary pseudonaturally in $A$ and $B$.

\begin{defn}
  For $X\in \sex(B)$ and $u:A\to B$, let $\omega_{X,u} : \widetilde{u^*X} \to u^* \Xtil$ be the natural transformation defined on objects by
  \begin{align*}
    \omega_{X,u}(a,\al,x) &= (ua,u\al,x)\\
    \omega_{X,u}(a,\al,x,a',\al',x',\xi) &= (ua,u\al,x,ua',u\al',x',\xi).
  \end{align*}
\end{defn}

\begin{lem}\label{thm:om-laxnat}
  Let $X,Y\in \sex(C)$ and $A\xto{v} B \xto{u} C$, and $f:X\to Y$ a morphism representative.
  Then the map $\omega_{X,1_A} : \Xtil\to \Xtil$ is equal to $1_{\Xtil}$, and the following diagrams commute:
    \[
      \begin{tikzcd}
        \widetilde{v^*u^*X} \ar[r,"{\omega_{u^*X,v}}"] \ar[dr,"{\omega_{X,uv}}"'] &
        v^* (\widetilde{u^*X}) \ar[d,"{v^*\omega_{X,u}}"] \\
        & v^* u^* \Xtil
      \end{tikzcd}
      \qquad
      \begin{tikzcd}
        \widetilde{u^*X} \ar[r,"{\widetilde{u^*f}}"] \ar[d,"{\omega_{u^*X,u}}"'] &
        \widetilde{u^*Y} \ar[d,"{\omega_{Y,u}}"]\\
        u^*\Xtil\ar[r,"{u^*\ftil}"'] & u^*\Ytil.
      \end{tikzcd}
    \]
\end{lem}
\begin{proof}
  By inspection of the definitions.
\end{proof}

\begin{lem}\label{thm:om-sexeqv}
  The functor $\sum_a \omega_{X,u,a}$ is a \sex-equivalence over $\ob{A}$.
\end{lem}
\begin{proof}
  First, we must define $s:\ob{u^*\Xtil} \to \ob{\widetilde{u^*X}}$.
  The first kind of object of $(u^*\Xtil)_c$ is $(b,\be,x)$ for $\be : b\to uc$ and $x\in X_{b,0}$; we send this to $(c,1_c,X_{\be,0}(x))$ in $(\widetilde{u^*X})_c$.
  The second kind of object of $(u^*\Xtil)_c$ is $(b,\be,x,b',\be',x',\xi)$ for $\be:b\to b'$, $x\in X_{b,0}$, $\be':b'\to uc$, $x'\in X_{b',0}$, and $\xi\in X_{b',1}$ a witness that $X_{\be,0}(x) \sim x'$; we send this to $(c,1_c,X_{\be'\be,0}(x),c,1_c,X_{\be',0}(x'),m(X_{\be,\be'}(x),X_{\be',1}(\xi)))$ in $(\widetilde{u^*X})_c$, where $X_{\be,\be'}(x)$ is a functoriality witness of $X$.

  Second, we must send morphisms in $(u^*\Xtil)_c$ to zigzags in $(\widetilde{u^*X})_c$.
  We send a morphism $(b,\be,x,b',\be',x',\xi) \to (b,\be'\be,x)$ to the one-morphism zigzag
  \[ (c,1_c,X_{\be'\be,0}(x),c,1_c,X_{\be',0}(x'),X_{\be',1}(\xi)) \to
    (c,1_c,X_{\be'\be,0}(x)),
  \]
  and similarly we send a morphism $(b,\be,x,b',\be',x',\xi) \to (b',\be',x')$ to the one-morphism zigzag
  \[
    (c,1_c,X_{\be',0}(X_{\be,0}(x)),c,1_c,X_{\be',0}(x'),X_{\be',1}(\xi))
    \to
    (c,1_c,X_{\be',0}(x'))
  \]

  Third, we must relate each object of $(u^*\Xtil)_c$ by a zigzag to its roundtrip image.
  For $(b,\be,x)$, we have
  \[ (b,\be,x) \ot
    (b,\be,x,uc,1_{uc},X_{\be,0}(x),r(X_{\be,0}(x))) \to
    (uc,1_{uc},X_{\be,0}(x)),
  \]
  while for $(b,\be,x,b',\be',x',\xi)$ we have
  \begin{multline*}
  (b,\be,x,b',\be',x',\xi) \to
    (b,\be'\be,x) \ot
    \bullet \to
    (uc,1_{uc},X_{\be'\be,0}(x)) \\\ot
    (uc,1_{uc},X_{\be'\be,0}(x),uc,1_{uc},X_{\be',0}(x'),X_{\be',1}(\xi))
  \end{multline*}
  where the middle zigzag is as above.

  Fourth and finally, we must relate each object of $(\widetilde{u^*X})_c$ by a zigzag to its roundtrip image.
  For $(a,\al,x)$ we have
  \[ (a,\al,x) \ot
     (a,\al,x,c,1_c,X_{u\al,0}(x),r(X_{u\al,0}(x))) \to
     (c,1_c,X_{u\al,0}(x)),
   \]
   while for $(a,\al,x,a',\al',x',\xi)$ we have
   \begin{multline*}
   (a,\al,x,a',\al',x',\xi) \to
     (a,\al'\al,x) \ot
     \bullet \to
     (c,1_c,X_{u(\al'\al),0}(x)) \\ \ot
     (c,1_c,X_{u(\al'\al),0}(x),c,1_c,X_{u\al',0}(x'),m(X_{u\al,u\al'}(x),X_{u\al',1}(\xi)))
   \end{multline*}
   where again the middle zigzag is as above.
\end{proof}

\begin{prop}
  For any \sex-local left derivator \D,
  the functors $\odt$ of \cref{thm:odt-components} vary pseudonaturally in $A,B\in\cCat$.
  Therefore, they define a morphism of derivators
  \[ \odt : \sex \times \D \to \D. \]
\end{prop}
\begin{proof}
  For $u:A\to A'$ and $v:B\to B'$, we define the pseudonaturality constraint
  \[ u^*X \odt v^* M
    = (\widetilde{u^*X} \odot v^* M)
    \xiso{\omega} (u^*\Xtil \odot v^*M)
    \cong (u^*\times v)(\Xtil \odot M)
    = (u^*\times v)(X\odt M).
  \]
  The map induced by $\omega_{X,u}$ is an isomorphism by \cref{thm:om-sexeqv}, while the second isomorphism is the pseudofunctoriality of $\odot$.
  The axioms for a pseudonatural transformation follow from those of $\odot$ and \cref{thm:om-laxnat}.
\end{proof}

\begin{prop}\label{thm:odt-cc}
  The above-defined $\odt$ is cocontinuous in both variables.
\end{prop}
\begin{proof}
  Cocontinuity in the second argument follows from that of $\odot$.
  For cocontinuity in the first argument, by (the two-variable version of) \cref{thm:cocts-disc} it suffices to show that for $u:A\to I$ in \cCat, with $I$ discrete, and $X\in \sex(A)$ and $M\in \D(B)$, the transformation $(u\times 1)_!(X\odt M) \to u_! X \odt M$ is an isomorphism.

  Since $I$ is discrete, we can let $(u_!X)_i$ be the colimit of $X$ restricted to $A_i$ as constructed in \cref{thm:colim}, and put these together into a coherent diagram $u_!X$.
  We then have the adjunction unit $\eta : X \to u^*u_! X$, consisting of the injections into these colimits.
  The map we must show to be an isomorphism is the composite
  \begin{align*}
    (u\times 1)_!(\Xtil \odot M)
    &\xto{\etatil} (u\times 1)_!(\widetilde{u^* u_!X} \odot M)\\
    &\xto{\omega} (u\times 1)_!(u^* (\widetilde{u_!X}) \odot M)\\
    &\toiso (u\times 1)_!\,(u\times 1)^*(\widetilde{u_!X} \odot M)\\
    &\xto{\phantom{\omega}} \widetilde{u_!X} \odot M.
  \end{align*}
  Furthermore, the composite $\omega\etatil$ induces a map $\tint\omega\etatil$ on Grothendieck constructions:
  \[
    \begin{tikzcd}
      \tint \Xtil \ar[d,"p_{\Xtil}"'] \ar[r,"\mathord{\int}\omega\etatil"] & \tint \widetilde{u_!X} \ar[d,"p_{\widetilde{u_!X}}"] \\
      A \ar[r,"u"'] & I,
    \end{tikzcd}
  \]
  and the desired map can then be identified with
  \[ (u\times 1)_!\,(p_{\Xtil}\times 1)_!\, (\pi_2)^*M
    \toiso (p_{\widetilde{u_!X}}\times 1)_!\,(\tint\omega\etatil)_!\,(\pi_2)^*M
    \to (p_{\widetilde{u_!X}}\times 1)_!\,(\pi_2)^*M.
  \]
  where both projections $\tint\Xtil \times B \to B$ and $\tint\widetilde{u_!X} \times B \to B$ are denoted $\pi_2$.
  Therefore, as in the proof of \cref{thm:odot}, it will suffice to show that $\tint\omega\etatil: \tint \Xtil \to \tint \widetilde{u_!X}$ is a \sex-equivalence over $I$.\footnote{This explains our earlier comment that the failure of $\odot$ to be cocontinuous in its first variable is due to $\tint$ being an oplax colimit rather than a homotopy colimit.}

  The objects of $\tint\Xtil$ are those of $\Xtil_c$ for all $c\in A$, hence of the two forms $(a,\al,x)$ and $(a,\al,x,a',\al',x',\xi)$ as usual.
  But its morphisms incorporate the morphisms of $A$ according to the Grothendieck construction; thus we have
  \begin{equation}
    (a,\gm\al'\al,x) \ot (a,\al,x,a',\al',x',\xi) \to (a',\gm\al',x')\label{eq:groth-zigzag}
  \end{equation}
  for any $\gm:c\to c'$.

  Since $I$ is discrete, $\tint \widetilde{u_!X}$ is essentially (up to an inessential modification by $X_r$ witnesses) the simple construction of \cref{thm:Xtil-done} applied to $u_!X$.
  Thus, as objects it has both elements of $(u_!X)_0$, which are pairs $(a,x)$ with $x\in X_{a,0}$, and elements of $(u_!X)_{1}$.
  By construction of $u_!X$, the latter sort of element is a sequence
  \[\Xi = (a_0,x_0,\al_1,\xi_1,a_1,x_1,\al_2,\xi_2,\dots,\al_n,\xi_n,a_n,x_n),\]
  where each $x_k\in X_{a_k,0}$, and for each $k$ either
  \begin{itemize}
  \item $\al_k : a_{k-1}\to a_k$ and $\xi_k$ is a witness that $X_{\al_k,0}(x_{k-1}) \sim x_k$, or
  \item $\al_k : a_{k}\to a_{k-1}$ and $\xi_k$ is a witness that $X_{\al_k,0}(x_{k}) \sim x_{k-1}$.
  \end{itemize}
  Such a sequence then comes with morphisms to both $(a_0,x_0)$ and $(a_n,x_n)$.

  \def\toe{\tint\omega\etatil}
  The functor $\toe$ is defined on objects by
  \begin{align*}
    \toe(a,\al,x) &= (a,x)\\
    \toe(a,\al,x,a',\al',x',\xi) &= (a,x,\al,\xi,a',x').
  \end{align*}
  As always, we use the characterization of \cref{thm:sex-eqv}.

  First, to define a function $s : \ob{(\tint \widetilde{u_!X})} \to \ob{(\tint\Xtil)}$, we send $(a,x)$ to $(a,1_a,x)$, and a zigzag sequence $\Xi$ as above to $(a_0,1_{a_0},x_0)$.

  Second, we can send the morphism $\Xi \to (a_0,x_0)$ to the identity.
  Before deciding what to do with the morphism $\Xi \to (a_n,x_n)$, note that given $\al:a\to a'$ and $\xi$ a witness that $X_{\al,0}(x)\sim x'$, we have a zigzag
  \[ (a,1_a,x) \ot
    (a,1_a,x,a,1_a,x,r(x)) \to
    (a,\al,x) \ot
    (a,\al,x,a',1_{a'},x',\xi) \to
    (a',1_{a'},x')
  \]
  in which the second morphism uses the extra flexibility of~\eqref{eq:groth-zigzag}, with $\gm=\al$.
  Now by concatenating these zigzags, possibly reversed as necessary, we obtain a zigzag from $(a_0,1_{a_0},x_0)$ to $(a_n,1_{a_n},x_n)$ from any $\Xi$, which is what we needed.

  Third, we need to relate any object of $\tint \widetilde{u_!X}$ to its roundtrip image by a zigzag.
  But an object of the form $(a,x)$ is equal to its roundtrip image, while $\Xi$ comes with a basic morpism to its roundtrip image $(a_0,x_0)$.

  Fourth and finally, we need to relate any object of $\tint\Xtil$ to its roundtrip image.
  The roundtrip image of $(a,\al,x)$ is $(a,1_a,x)$, for which we have as above
  \[ (a,1_a,x) \ot
    (a,1_a,x,a,1_a,x,r(x)) \to
    (a,\al,x). \]
  And the roundtrip image of $(a,\al,x,a',\al',x',\xi)$ is $(a,1_a,x)$, for which we have the previous zigzag together with
  \[ (a,\al,x) \ot
    (a,1_a,x,a,\al,x,r(x)) \to
    (a,\al'\al,x) \ot
    (a,\al,x,a',\al',x',\xi)
  \]
  in which the middle morphism uses the extra flexibility of~\eqref{eq:groth-zigzag} with $\gm=\al'$.
\end{proof}

\begin{cor}\label{thm:extend-cchom}
  For any \sex-local left derivator \D and any $M\in \D(\done)$, there is a cocontinuous morphism $(-\odt M) : \sex \to \D$ such that $\ast\odt M \cong M$, where $\ast \in \sex(\done)$ is the terminal object.
\end{cor}
\begin{proof}
  It remains to show that $\ast\odt M \cong M$.
  By definition, $\ast\odt M = \widetilde{\ast}\odot M$, where $\widetilde{\ast}$ is $(\cdot \toto\cdot)$.
  But the functor $\widetilde{\ast} \to \done$ is, as noted previously, a \sex-equivalence.
  Thus the induced map $\widetilde{\ast}\odot M\to M$ is an isomorphism, since \D is \sex-local.
\end{proof}

\begin{thm}\label{thm:sex-rfcc}
  If \D is a \sex-local left derivator, then the functor
  \[ \cchom(\sex,\D) \to \D(\done), \]
  induced by evaluation at $\ast\in \sex(\done)$, is an equivalence of categories.
  In other words, \sex is a relative free cocompletion of a point.
\end{thm}
\begin{proof}
  The construction of \cref{thm:extend-cchom} is functorial and the isomorphism is natural.
  Thus, it suffices to construct, for any cocontinuous $G:\sex\to\D$, an isomorphism $G X \cong X\odt G(\ast)$, natural in $G$ and in $X\in \sex(A)$.
  For this we have
  \begin{align*}
    GX &\cong G(\Xtil \odot \ast)\\
    &= G((p_{\Xtil} \times 1)_!\, (\pi_2)^* (\ast))\\
    &\cong (p_{\Xtil} \times 1)_!\, (\pi_2)^* G(\ast)\\
    &= \Xtil \odot G(\ast)\\
    &= X \odt G(\ast),
  \end{align*}
  where the first isomorphism is \cref{thm:self-odt}, and the second is because $G$ is cocontinuous.
  Naturality in $G$ is evident, while naturality in $X$ follows from the second part of \cref{thm:self-odt}.
\end{proof}

\section{Other relative free cocompletions}
\label{sec:other-rfc}

Once we have one relative free cocompletion --- in our case, \sex --- it is much easier to construct other \sex-local ones.
First we note that if \D is distributive (\cref{defn:distrib}), then the whole two-variable morphism $\odt:\sex\times \D\to\D$ is determined by the functor $L:\sex\to \D$ defined by
\[ L X = X\odt \ast . \]

\begin{lem}\label{thm:distrib-odt}
  If \D is distributive and \sex-local, we have a natural isomorphism
  \[ X \odt M \cong LX \times M \]
  for $X\in \sex(A)$ and $M\in \D(B)$.
\end{lem}
Here $\times$ on the right-hand side denotes the functor $\D(A)\times \D(B) \to \D(A\times B)$ induced by the cartesian product of \D.
\begin{proof}
  By definition,
  \begin{align*}
    X\odt M 
    &= (p_{\Xtil}\times 1)_!\, \pi_2^* (M) \\
    &\cong (p_{\Xtil}\times 1)_!\, \pi_2^* (\ast \times M) \\
    &\cong (p_{\Xtil}\times 1)_!\, \pi_2^* (\ast) \times M \qquad \text{(by distributivity)}\\
    &= (X\odt \ast) \times M\\
    &= LX \times M.\qedhere
  \end{align*}
\end{proof}

\begin{cor}\label{thm:distrib-loc-eqv}
  If \D is distributive and \sex-local, and $f:X\to Y$ is a morphism representative in $\sex(A)$ such that $Lf$ is an isomorphism, then $\sum_a \ftil_a$ is a \D-equivalence over $\ob{A}$.
\end{cor}
\begin{proof}
  By \cref{thm:distrib-odt}, the assumption implies that $f\odt M$ is an isomorphism for any $M\in \D(B)$.
  In particular, for $M\in \D(\ob{A})$ the induced map
  \[ (p_{\Xtil}\times 1)_!\, \pi_2^* (M) \to (p_{\Ytil}\times 1)_!\, \pi_2^* (M) \]
  is an isomorphism, where the functors fit into the diagram on the left of \cref{fig:dle}.

  The two functors $p_{\Xtil}\times 1$ and $p_{\Ytil}\times 1$ are split opfibrations, and the pullback of $\ftil\times 1$ along $(\incl{A}\times 1_{\ob{A}}) : \ob{A} \to A\times\ob{A}$ is $\sum_a \ftil_a$.
  Thus, the corresponding map for the diagram on the right of \cref{fig:dle} is also an isomorphism; but this is precisely to say that $\sum_a \ftil_a$ is a \D-equivalence over $\ob{A}$.
\end{proof}

\begin{figure}
  \[
    \begin{tikzcd}
      \tint\Xtil\times \ob{A} \ar[dr] \ar[ddr,"{p_{\Xtil}\times 1}"'] \ar[drr,"\pi_2"] \\
      & \tint\Ytil\times \ob{A}  \ar[d,"{p_{\Ytil}\times 1}"] \ar[r,"\pi_2"'] & \ob{A}\\
      & A\times \ob{A}
    \end{tikzcd}\quad
    \begin{tikzcd}
      \sum_a \Xtil_a  \ar[dr] \ar[ddr,""'] \ar[drr,""] \\
      & \sum_a \Ytil_a  \ar[d,""] \ar[r,""'] & \ob{A}\\
      & \ob{A}
    \end{tikzcd}    
  \]
  \caption{Diagrams for the proof of \cref{thm:distrib-loc-eqv}}
  \label{fig:dle}
\end{figure}
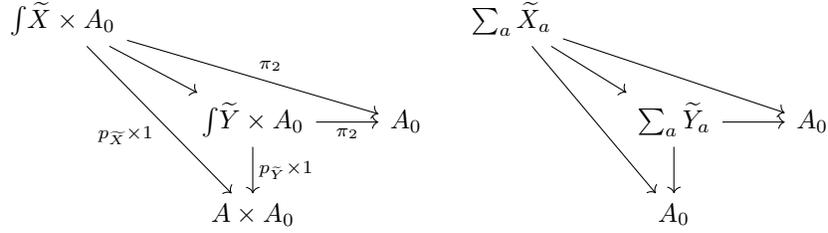

\begin{thm}\label{thm:subsex-rfcp}
  If \T is \sex-local and distributive, and $L:\sex\to\T$ has a right adjoint with invertible counit, then \T is a relative free cocompletion of a point.
\end{thm}
\begin{proof}
  Let \D be a \T-local left derivator; we must show that the precomposition functor
  \( (-\circ L) : \cchom(\T,\D) \to \cchom(\sex,\D) \)
  is an equivalence.
  We have a commutative square
  \[
    \begin{tikzcd}
      \cchom(\T,\D) \ar[r,"(-\circ L)"] \ar[d] & \cchom(\sex,\D) \ar[d] \\
      \hom(\T,\D) \ar[r,"(-\circ L)"'] & \hom(\sex,\D)      
    \end{tikzcd}
  \]
  in which the vertical functors are fully faithful.
  But the bottom functor has a left adjoint $(-\circ R)$, where $R$ is the right adjoint of $L$, with invertible unit, and hence is also fully faithful.
  Thus the top functor is also fully faithful.
  So it suffices to show it is split essentially surjective, i.e.\ that any cocontinuous $G:\sex\to\D$ factors through $L$, up to isomorphism, by a specified cocontinuous morphism.

  To start with, we have a canonical morphism $GR : \T\to\D$.
  We also have a unit map $\eta:1_{\sex} \to RL$, and since the counit of the adjunction is invertible, $L\eta$ is an isomorphism.
  Thus, by \cref{thm:distrib-loc-eqv}, for any $X\in\sex(A)$, if we choose a representative for $\eta_X$, then $\sum_a (\widetilde{\eta_X})_a$ is a \T-equivalence over $\ob{A}$.
  Since \D is \T-local, this means it is also a \D-equivalence.
  And since $G$ is of the form $(-\odt M)$ for some $M\in \D(\done)$, by \cref{thm:sex-rfcc}, it follows that $G$ also inverts $\eta_X$.
  In other words, $G\eta$ is an isomorphism $G\cong GRL$.

  It remains to show that $GR$ is cocontinuous.
  This means to show that the mate $u_! GR \to GR u_!$ of the isomorphism $GR u^* \cong u^* GR$ is again an isomorphism.
  The latter isomorphism is the pasting composite of the following squares:
  \[
    \begin{tikzcd}
      \T(A) \ar[r,"R"]  \ar[dr,phantom,"\cong"] & \sex(A) \ar[r,equals] \ar[dr,phantom,"\cong"] & \sex(A) \ar[r,"G"] \ar[dr,phantom,"\cong"] & \D \\
      \T(B) \ar[r,equals] \ar[u,"u^*"] & \T(B) \ar[u,"R u^*"'] \ar[r,"R"'] & \sex(B) \ar[u,"u^*"'] \ar[r,"G"'] & \D \ar[u,"u^*"']
    \end{tikzcd}
  \]
  Therefore, by the functoriality of mates, its mate is the pasting composite of the following squares:
  \[
    \begin{tikzcd}
      \T(A) \ar[r,"R"] \ar[d,"u_!"] \twocell{dr} & \sex(A) \ar[d,"u_!L"] \ar[r,equals] \twocell{dr} & \sex(A) \ar[d,"u_!"] \ar[r,"G"] \twocell{dr} & \D \ar[d,"u_!"] \\
      \T(B) \ar[r,equals] & \T(B) \ar[r,"R"'] & \sex(B) \ar[r,"G"'] & \D 
    \end{tikzcd}
  \]
  The left-hand square is the counit $LR \to 1_\T$, which is an isomorphism by assumption.
  The right-hand square is an isomorphism since $G$ is cocontinuous.
  Finally, the middle square is the unit $1_{\sex}\to RL$, which as we just showed is inverted by $G$.
  Thus, the pasting composite is also an isomorphism, so $GR$ is cocontinuous.
\end{proof}

\begin{rmk}
  If we omit the hypothesis of distributivity in \cref{thm:subsex-rfcp}, the same argument implies that \T is a \emph{localization} of \sex, in the sense that the precomposition functor $\cchom(\T,\D) \to \cchom(\sex,\D)$ is fully faithful, and its full image consists of the morphisms $\sex\to \D$ that invert the same morphisms that are inverted by $L:\sex\to\T$.
  (More abstractly, this can be expressed as a \emph{coinverter} in the 2-category of derivators: a 2-categorical colimit that universally forces some 2-cell to become invertible.)
  Distributivity enables us to reexpress this as \T being a relative free cocompletion of a point, without explicit reference to $L:\sex\to \T$.
\end{rmk}

We have already observed that all the derivators in \cref{fig:locality} are \sex-local and distributive.
Thus, it suffices to show that their $L$-functors all have right adjoints.

\begin{eg}
  For $\T=\iSet$, $L$ computes the quotient of each pseudo-equivalence relation in a coherent diagram, obtaining an ordinary diagram of sets.
  This has a right adjoint that assigns to any set the identity (pseudo-)equivalence relation on it, of which it is the quotient.
  Thus, \iSet is a relative free cocompletion of a point.
\end{eg}

\begin{eg}
  For $\T=\iProp$, $L$ computes the support $\pi_{-1}(X_0)$ of each pseudo-equivalence relation in a coherent diagram.
  Since the quotient of a pseudo-equivalence relation is inhabited if and only if $X_0$ is, this factors through $\iSet$ via the standard support functor $\iSet\to\iProp$.
  The latter has a right adjoint assigning to each proposition the corresponding subsingleton, which is its own support.
  Thus, \iProp is a relation free cocompletion of a point.
\end{eg}

We leave the trivial case $\T=\iContr$ to the reader.

\begin{eg}
  For $\T=\spos$, $L$ sends each pseudo-equivalence relation to $X_0+X_1$, which is isomorphic in \spos to $X_0$.
  This has a right adjoint that sends each object $X$ of \spos to the \emph{full} (pseudo-)-equivalence relation on it, i.e.\ $X_0 = X$ and $X_1 = X\times X$.
  The counit is evidently an isomorphism, so \spos is a free cocompletion of a point.
\end{eg}

\begin{eg}
  Finally, for $\T=\sreg$, $L$ sends each pseudo-equivalence in a coherent diagram to its image, which is an actual equivalence relation.
  This has a right adjoint that sends each equivalence relation to itself, regarded as a pseudo-equivalence relation.
  Thus, \sreg is also a free cocompletion of a point.
\end{eg}

\section{Conclusions and speculations}
\label{sec:anafunctors}

We have constructed three different relative free cocompletions of a point, \iSet, \sreg, and \sex, which are nevertheless all intuitively ``1-categorical''.
Similarly, both \iProp and \spos are intuitively ``0-categorical'' (i.e.\ posetal).
Thus we may reasonably wonder, what happens in higher dimensions?
The obvious candidate for a 2-categorical (or, more precisely, $(2,1)$-categorical) relative free cocompletion of a point is a derivator of groupoids; but we have multiple notions of groupoid.

On the one hand, we have the standard notion of groupoid, with hom-sets.
These should yield a derivator \iGpd: the objects of $\iGpd(A)$ are pseudofunctors $A\to \iGpd$, and its morphisms are isomorphism classes of pseudonatural transformations.
In particular, the isomorphisms in the derivator \iGpd would be the \emph{equivalences} of groupoids, in the usual constructive sense with a specified pseudo-inverse functor.

On another hand, we can consider \emph{\cE-groupoids}, ``groupoids enriched over setoids'' (see e.g.~\cite{hs:constructive-ct,bd:tt-lcc-int} for \cE-categories).
These should yield a derivator \iEGpd.
And there is a third notion in between, of groupoids enriched over equivalence relations, which should yield a derivator \iRGpd.
It seems likely that we should have an analogous three notions of $n$-groupoid for all finite $n$, where the top level is enriched either over \iSet, \sreg, or \sex.
But in the limit $n\to \infty$, where there is no longer a ``top level'', it seems reasonable to expect the difference to disappear, so that there would be only one absolute free cocompletion of a point \iSpace.

\begin{conj}
  One can constructively define an absolute free cocompletion of a point using some kind of cubical sets, simplicial sets, or semisimplicial sets, along with three reflective localizations of it for each finite $n$, consisting of the $n$-groupoids enriched over sets, setoids, and equivalence relations at the top dimension.
\end{conj}

However, something funny happens with the locality preorder at dimension 2.
Just as the \iSet-equivalences are the functors inducing an isomorphism under the reflection $\pi_0$ of categories into sets, we expect the \iGpd-equivalences should be the functors inducing an equivalence under the reflection $\Pi_1$ of categories into groupoids.
But since $\Pi_1(A)$ has the same set of objects as $A$, if $f:A\to B$ is a \iGpd-equivalence then we have an actual function $\ob{B}\to \ob{A}$, suggesting that a \iGpd-equivalence should also be not just a \iSet-equivalence but a \sreg-equivalence.
Thus \sreg should be \iGpd-local, and similarly we expect \sex to be \iRGpd-local, leading to the placements of \iGpd, \iRGpd, and \iEGpd in the extension of \cref{fig:locality} shown in \cref{fig:more-locality}.

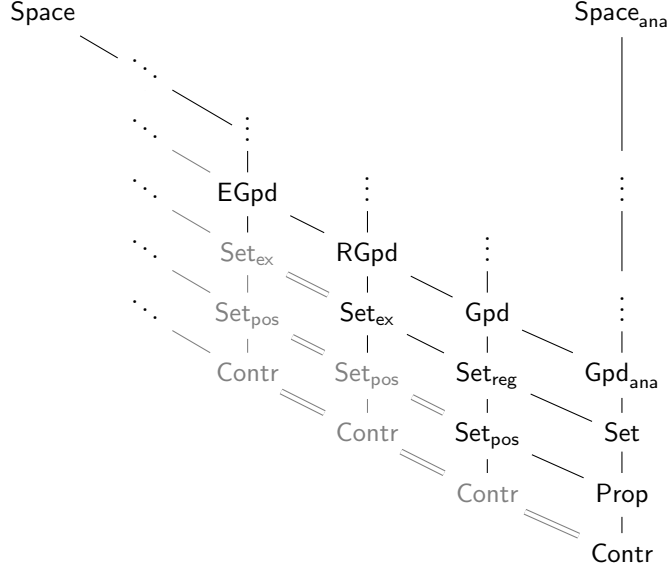
\begin{figure}
  \centering
  \[
  \begin{tikzcd}[row sep={.8cm,between origins},column sep=small]
    \iSpace \ar[dr,-] &&&&& \sana \ar[ddd,-]\\
    & \ddots \ar[d,-] \ar[dr,-] \\
    & \ddots \ar[dr,gray,-] & \vdots \ar[d,-] \\
    &\ddots \ar[dr,gray,-] &\iEGpd \ar[dr,-] \ar[d,-] & \vdots \ar[d,-]  & & \vdots \ar[dd,-] \\
    &\ddots \ar[dr,gray,-]&{\color{gray}\sex} \ar[dr,equals,gray] \ar[d,-,gray] &\iRGpd \ar[d,-] \ar[dr,-]& \vdots \ar[d,-] \\
    & \ddots \ar[dr,gray,-] &{\color{gray}\spos} \ar[d,gray,-] \ar[dr,equals,gray] &\sex \ar[dr,-] \ar[d,-] & \iGpd \ar[d,-] \ar[dr,-]& \vdots \ar[d,-]  \\
    &&{\color{gray}\iContr} \ar[dr,gray,equals]&{\color{gray}\spos} \ar[d,gray,-] \ar[dr,gray,equals] &\sreg \ar[dr,-] \ar[d,-] & \iGpd\ana \ar[d,-]\\
    &&&{\color{gray}\iContr} \ar[dr,gray,equals]&\spos \ar[d,-]\ar[dr,-] & \iSet \ar[d,-]\\
    &&&& {\color{gray}\iContr} \ar[dr,gray,equals] & \iProp \ar[d,-] \\
    &&&&& \iContr
  \end{tikzcd}
\]
  \caption{A conjectural enlargement of \cref{fig:locality}}
  \label{fig:more-locality}
\end{figure}

The diagonal rows\footnote{They are diagonal rather than horizontal, of course, so that the picture is still a sort of ``Hasse diagram'' of the locality relation (although we do not mean to exclude the possible existence of further intermediate objects not drawn).} of this diagram are at constant ``categorical dimension'' while moving vertically downwards passes to the subcategory of truncated objects.
That is, the categories of subterminal objects in \iSet and \sreg are equivalent to \iProp and \spos respectively, and we expect the categories of 0-truncated objects in \iGpd and \iRGpd to be equivalent to \sreg and \sex respectively.
Since \spos is also the category of subterminal objects in \sex, and \sex should also be the category of 0-truncated objects in \iEGpd, it is natural to extend the diagram further to the left in a way that ``stabilizes'' after a certain number of steps, as we have done in gray.
One can thus view ``exact completion'' as adding an additional dimension to the Baez--Dolan ``periodic table of $n$-categories''~\cite{bd:hda-tqft}, which stabilizes along the $n$-categorical row at the $(n+2)^{\mathrm{nd}}$ stage.

It is worth noting that the derivators in the ``middle'' of this diagram, though like all the others they are relative free cocompletions of a point, are not as well-endowed with exactness properties.
For instance, \iSet and \sex are both exact, but \sreg is not: an internal equivalence relation in \sreg is a \emph{pseudo}-equivalence relation in \iSet, but it can only be effective in \sreg if it is an actual equivalence relation.
Similarly, but perhaps more surprisingly, \iGpd is not exact as a $(2,1)$-category (in a sense like that of~\cite{street:stacks}): for if it were, its subcategory of 0-truncated objects would be exact as a 1-category, but this subcategory is \sreg.

I expect \iRGpd to also fail to be $(2,1)$-exact, though less obviously since its subcategory of 0-truncated objects should be \sex, which is 1-exact.
But \sex should also be the subcategory of 0-truncated objects in \iEGpd, which should be $(2,1)$-exact.
This is analogous to how \spos is ``$(0,1)$-exact'' (i.e.\ a distributive lattice), and is the subcategory of subterminal objects in both \sex and \sreg, though only the former is 1-exact.

Is there a different 2-dimensional relative free cocompletion of a point whose category of 0-truncated objects is \iSet?
To guess what this might be, note that in the parts of \cref{fig:more-locality} that we understand precisely so far, moving to the right can be achieved by passing to a localization.
For instance, if we localize \spos by inverting the surjections, we obtain \iProp.
Similarly, if in \sex we invert the morphisms $f:X\to Y$ that reflect equality (in the sense that if there exists a witness that $f_0(x)\sim f_0(x')$ then there exists a witness that $x\sim x'$) and such that $f_0$ is split surjective, we obtain \sreg.
If we further invert the morphisms that reflect equality and such that $f_0$ is merely surjective, we obtain $\iSet$.

Analogously, it is natural to guess that \iRGpd should be obtainable from \iEGpd by inverting functors that are split-surjective on objects, split-full on morphisms, and reflect equality of parallel morphisms; while \iGpd should be similarly obtainable from \iRGpd by inverting functors that are split-surjective on objects, merely full on morphisms, and reflect equality of parallel morphisms.
This suggests that the ``missing link'' should be obtained from \iGpd by inverting the functors that are fully faithful and merely surjective on objects.
This is equivalent to inverting the \emph{weak equivalences}: functors that are fully faithful and essentially surjective.%
\footnote{Recall that every weak equivalence is an equivalence if and only if the axiom of choice holds.}
The morphisms in this localization are \emph{anafunctors}~\cite{makkai:avoiding-choice,bartels:hgt,roberts:ana,roberts:elem-ana}, so we denote it $\iGpd\ana$.

Similarly, if we present \iSet as a localization of \sex, we could call its morphisms \emph{anafunctions} and write $\iSet \simeq (\sex)\ana$.
Equivalently, we can observe that since \iSet is already exact, it is equivalent to its own exact completion \emph{as a regular category}, i.e.\ $\iSet \simeq \iSet_{\mathsf{ex/reg}}$; in general we can present the ex/reg completion as consisting of setoids or equivalence relations with anafunctions between them (``total and functional relations'').
This suggests that the missing link $\iGpd\ana$ should be the ``$(2,1)$-exact completion of $\iSet$ as a regular category''.
This makes sense because the definition of $\iGpd\ana$, unlike that of $\iGpd$, incorporates some information about the regular structure of \iSet, i.e.\ the surjective functions of sets.


There are, however, issues with actually performing the localization leading to the hypothetical $\iGpd\ana$.
In particular, unlike \iRGpd and \iGpd, it is not a \emph{reflective} localization of \iEGpd.
Worse, even in ZF set theory, with excluded middle but no choice, it is impossible to prove that $\iGpd\ana$ is locally small, cartesian closed, or complete~\cite{ak:nonsmall-ana}, and hence it seems unlikely to be a derivator.
(This also implies that it cannot be presented by any sort of model category, although weaker structures like a fibration or cofibration category are a possibility.)
However, it may be easier to obtain at least a \emph{left} derivator of this sort, with colimits but not necessarily limits.

\begin{conj}
  There is a left derivator $\iGpd\ana$ composed of groupoids and anafunctors.
  Moreover:
  \begin{itemize}
  \item $\iGpd\ana$ is a relative free cocompletion of a point, and is ``$(2,1)$-exact''.
  \item Every weak equivalence of categories is a $\iGpd\ana$-equivalence.
  \item \iSet is $\iGpd\ana$-local, but \sreg and \spos are not.
  \item The subcategory of 0-truncated objects in $\iGpd\ana$ is \iSet.
  \end{itemize}
\end{conj}

Of course, we can ask analogous questions about $n$-groupoids for $2\le n\le \infty$.

\begin{conj}
  There is a left derivator $\sana$ composed of ``$\infty$-groupoids and anafunctors''.
  Moreover:
  \begin{itemize}
  \item $\sana$ is a relative free cocompletion of a point, and is ``$(\infty,1)$-exact''.
  \item Every weak equivalence of categories is an \sana-equivalence.
  \item \iSet and $\iGpd\ana$ are $\sana$-local, but \spos, \sreg, and \iGpd are not.
  \item The subcategory of 1-truncated objects in \sana is $\iGpd\ana$.
  \end{itemize}
\end{conj}

These conjectural derivators $\iGpd\ana$ and $\sana$ are closely related to the issue raised in \cref{sec:introduction} that perhaps our definition of derivator is wrong: maybe we should use $\cCat\ana$ instead of $\cCat$.\footnote{It seems that replacing $\cCAT$ by $\cCAT\ana$ makes less of a difference.
  Since functors are in particular anafunctors, all our examples such as \sex are still derivators with this generalized definition.
  And as long as \emph{all} the functors $u^*,u_!,u_*$ in the target \D, and the components of derivator morphisms, are generalized to anafunctors simultaneously, I would expect essentially the same arguments for their universality to go through.}
Since $\cCat\ana$ is equivalent to the bicategory obtained by inverting the weak equivalence functors in \cCat, a natural definition of \textbf{ana-derivator} would be simply as a derivator such that $u^*:\D(B)\to\D(A)$ is a (perhaps weak) equivalence whenever $u:A\to B$ is a weak equivalence.

Of the derivators considered in this paper, \iSet and \iProp are ana-derivators, while it seems that the others are not (though I do not have a formal proof).
For instance, let $u:A\to B$ be a weak equivalence functor with $B$ discrete, and $X\in \spos(B)$.
Then $(u_* u^* X)_b$ is the power $X_b^{\ob{u^{-1}(b)}}$ of the set $X_b$ by the objects in the $u$-preimage of $b$.
The adjunction unit $X \to u_* u^* X$ consists of the diagonals $X_b \to X_b^{\ob{u^{-1}(b)}}$, but there seems no way to define a family of functions in the other direction without choosing elements of the fibers to give factors to project onto.

\begin{conj}
  $\iGpd\ana$ and $\sana$ are left ana-derivators.
  Moreover, \sana is the free cocompletion of a point among ana-derivators, while $\iGpd\ana$ is a relative free cocompletion of a point therein.
\end{conj}

\begin{rmk}
  It is natural to wonder, if the right-hand column in \cref{fig:more-locality} has its ``own notion of derivator'' (the above-defined ana-derivators), why is that not the case for the other columns?
  In fact, there are other ways to vary the notion of derivator.
  The notion of derivator we have worked with corresponds roughly to the second column from the right; but one could also replace the 2-categories \cCat and/or \cCAT by \cE-2-categories of \cE-categories, or \cR-2-categories of \cR-categories.%
  \footnote{To continue getting new notions beyond the fourth column, one would need to generalize to ``$n$-derivators'' in the sense of~\cite{raptis:hderiv}, with the domain \cCat replaced by some version of $(n,1)\text{-}\cCat$.
  That is, the notion of derivator can vary not only with the column but also with the row.}
  I have not pursued this direction; the goal of this paper was to show that \emph{even} if we try as hard as possible to take sets and set-based categories as our basic notions, we seem to be led, ineluctably, either to setoids and \cE-groupoids, or to anafunctors.

  The next question is, if only the right-hand column of \cref{fig:more-locality} consists of $\cCat\ana$-derivators, why does the \emph{whole} figure consist of \cCat-derivators, rather than just the two right-hand columns?
  In fact, I would expect that if we define \sex (for instance) as an \cECat-derivator, it would \emph{not} be a ``\cCat-derivator'' in the sense that $u^*$ is an equivalence for any \cE-functor $u$ that is inverted by the reflection of \cECat into \cCat.
  The difference is that \cCat is a \emph{reflective} localization of \cECat, so that we can make the \cECat-derivator \sex into a \cCat-derivator in a different way by simply \emph{restricting} its domain to the sub-2-category \cCat of \cECat.
  The latter restriction is the derivator we have called \sex in this paper.
\end{rmk}

A positive solution to the above conjectures would, I believe, give a systematic explanation of many confusing aspects of homotopy theory in set-based constructive mathematics.
However, it is not clear whether it would conclusively answer the question of what the ``correct'' constructive theory of spaces is, since both candidates $\iSpace$ and $\sana$ have drawbacks: the former truncates to $\sex$ rather than \iSet, while the latter is not locally small, cartesian closed, or complete.

Of course, such bifurcations of classical notions are not uncommon in constructive mathematics.
However, in this case there is more to be said: if we are willing to modify the background theory (while still keeping it ``constructive'' in at least some sense), we can make $\iGpd\ana$ and $\sana$ much better-behaved.

It is known that local smallness and cartesian closure of $\iGpd\ana$ (and also, presumably, $\sana$) requires much less than the full axiom of choice: it suffices to assume SCSA~\cite{makkai:avoiding-choice} or WISC~\cite{roberts:ana} (a.k.a.\ AMC~\cite{vdb:pred-topos}).
These weak choice axioms have at least some claim to being constructive, as they often hold in large classes of models of constructive mathematics, such as Grothendieck toposes, realizability toposes, and exact completions.
I do not know whether these axioms make $\iGpd\ana$ complete, but there is another axiom that should do so: the Axiom of Stack Completions~\cite{bh:psmon-stcplt}, which implies that $\iGpd\ana$ is equivalent to a \emph{reflective} localization of \iGpd (hence also of \iEGpd), whose objects are the ``intrinsic stacks'' relative to surjections of sets.
The constructive nature of ASC is perhaps debatable, but at least it holds in all Grothendieck toposes~\cite{jt:strong-stacks}.

Another approach is to choose instead to do constructive homotopy theory based on a foundational system in which spaces are primitive objects, such as homotopy type theory.
This is my preferred solution, so I will conclude with some remarks about its advantages.

\begin{rmk}\label{rmk:univalence}
  As noted in~\cite{lumsdaine:ecats}, the diagonals of \cref{fig:more-locality} bear a strong resemblance to the hierarchy of \emph{saturation} or \emph{univalence} conditions on (higher-)categorical structures defined in homotopy type theory~\cite{aks:rezk,anst:up}.
  When a groupoid is presented by a diagram on an inverse-category signature as in~\cite{makkai:folds,anst:up}, it has three ranks of type dependency, corresponding to the objects, morphisms, and equalities.
  Roughly speaking, \cE-groupoids correspond to unrestricted categories of this sort, while \cR-groupoids are univalent at the top rank (equalities), and ordinary groupoids are univalent at the top two ranks (equalities and morphisms).

  In a set-based foundation, it is not possible to be more univalent than this; but in homotopy type theory, we can also impose univalence conditions at the bottom rank of objects.
  The resulting homotopy theory \iUGpd of \emph{univalent groupoids} is a \emph{reflective} localization of \iGpd\footnote{For the expert, note that here we interpret ``groupoids'' as particular \emph{precategories} in the sense of~\cite{aks:rezk,hottbook}, with no dimension restriction on their type of objects.} at the weak equivalences, closely related to the category of ``intrinsic stacks'' mentioned above in connection with ASC.
  Hence, \iUGpd plays a similar role to $\iGpd\ana$, but without the attendant disadvantages.
  In particular, it is locally small, cartesian closed, exact, and has limits as well as colimits, while its subcategory of 0-truncated objects is \iSet.
  Similarly, the category of univalent $\infty$-groupoids (spaces) plays the expected role of $\iSpace\ana$.

  In fact, a ``univalent groupoid'' is equivalently just a type with the property of being a 1-type, while a ``univalent space'' is simply a type with no restrictions.
  That is, in homotopy type theory the primitive objects are the objects of $\iSpace\ana$ rather than those of \iSet, so that none of the elaborate work involved in defining higher groupoids and homotopy spaces is necessary.
  (The related notions of higher \emph{category}, however, are still nontrivial.)

  I expect that the primitive spaces in homotopy type theory form a derivator (although proving this may require an enhanced theory such as~\cite{ack:2ltt}).
  It is unclear whether the resulting derivator of univalent spaces would be a free cocompletion of a point;
  the answer might depend on how univalent the 1-categories in \cCat are assumed to be, and/or on strong classicality axioms such as
  $\mathsf{AC}_{\infty,-1}$ from~\cite[Exercise 7.8]{hottbook}.
  (In particular, since univalent 1-categories are now a reflective localization of non-univalent ones, it seems likely that all the other derivators in \cref{fig:more-locality} will still exist even if we replace \cCat by \cUCat.
  Thus $\iSpace\ana$ may not be a free cocompletion of a point unless there is a classicality axiom to collapse the columns.)

  However, the ``correctness criterion'' advanced in this paper for a homotopy theory of spaces is not justified for homotopy type theory anyway.
  This criterion seeks to characterize the homotopy theory of spaces in terms of sets (or at most 1-categories); thus it makes sense in a world whose primitive objects are sets, but not in a world where spaces are already present as primitive objects.
\end{rmk}

\bibliographystyle{alpha}
\bibliography{all}

\end{document}